\documentclass{amsproc}

\usepackage{amsmath}
\usepackage{amssymb}
\usepackage{graphicx}
\DeclareGraphicsExtensions{.png,.pdf,.eps}
\usepackage[all,2cell]{xy}

\newtheorem{theorem}{Theorem}[section]
\newtheorem{lemma}[theorem]{Lemma}
\newtheorem{corollary}[theorem]{Corollary}
\newtheorem{proposition}[theorem]{Proposition}
\newtheorem{conjecture}[theorem]{Conjecture}

\theoremstyle{definition}

\newtheorem{definition}[theorem]{Definition}

\newtheorem{remark}[theorem]{Remark}

\newtheorem{set}[theorem]{Setup}
\newtheorem{example}[theorem]{Example}

\def\B{{\mathbb B}}

\def\G{{\mathbb G}}

\def\P{{\mathbb P}}
\def\Q{{\mathbb Q}}

\def\Z{{\mathbb Z}}

\def\cC{{\mathcal C}}

\def\cE{{\mathcal E}}
\def\cF{{\mathcal F}}

\def\cI{{\mathcal I}}

\def\cM{{\mathcal M}}
\def\cN{{\mathcal{N}}}
\def\cO{{\mathcal{O}}}
\def\cP{{\mathcal{P}}}
\def\cQ{{\mathcal{Q}}}

\def\cS{{\mathcal S}}

\def\cU{{\mathcal U}}

\def\Q{{\mathbb{Q}}}
\def\G{{\mathbb{G}}}

\def\f{\varphi}

\def\f{\varphi}

\def\lra{\longrightarrow}

\def\ra{\rightarrow}
\def\lra{\longrightarrow}
\def\rat{\dashrightarrow}

\def\til{\widetilde}

\def\operatorname#1{\mathop{\rm #1}\nolimits}

\def\Proj{\operatorname{Proj}}

\def\Chow{\operatorname{Chow}}

\def\Exc{\operatorname{Exc}}

\def\Hom{\operatorname{Hom}}

\def\Cloc{\operatorname{Chlocus}}

\def\Pic{\operatorname{Pic}}

\def\im{\operatorname{im}}

\def\det{\operatorname{det}}

\def\loc{\operatorname{Locus}}
\def\rat{\operatorname{RatCurves}}

\def\NE{{\operatorname{NE}}}
\def\ME{{\operatorname{ME}}}
\def\Nef{{\operatorname{Nef}}}

\def\Eff{{\operatorname{Eff}}}
\newcommand{\cNE}[1]{\overline{\NE(#1)}}
\newcommand{\cME}[1]{\overline{\ME(#1)}}
\newcommand{\cEff}[1]{\overline{\Eff(#1)}}
\def\ev{{\operatorname{ev}}}
\def\arg{{\operatorname{arg}}}
\def\rel{{\operatorname{rel}}}
\DeclareMathOperator{\cloc}{\mathrm{ChLocus}}
\DeclareMathOperator{\ch}{\mathrm{Chain}}

\newcommand{\shse}[3]{0 ~\ra ~#1~ \lra ~#2~ \lra ~#3~ \ra~ 0}

\makeindex
\begin{document}
%\pagewiselinenumbers

\title{On rank $2$ vector bundles on Fano manifolds}

\author{Roberto Mu\~noz}
\address{Departamento de Matem\'atica Aplicada, ESCET, Universidad
Rey Juan Carlos, 28933-M\'ostoles, Madrid, Spain}
%\curraddr{Department of Mathematics and Statistics,
%Case Western Reserve University, Cleveland, Ohio 43403}
\email{roberto.munoz@urjc.es}
\thanks{First and third author partially supported by the spanish government project MTM2009-06964, second author partially supported by MIUR (PRIN project: Propriet\`a
geometriche delle
variet\`a reali e complesse)}

\author{Gianluca Occhetta}
\address{Dipartimento di Matematica, Universit\'a di Trento, via
Sommarive 14 I-38123 Povo (TN), Italy}
\email{gianluca.occhetta@unitn.it}
%\thanks{Partially supported by MIUR (PRIN project: Propriet\`a
%geometriche delle
%variet\`a reali e complesse)}

%    Information for third author
\author{Luis E. Sol\'a Conde}
\address{Departamento de Matem\'atica Aplicada, ESCET, Universidad
Rey Juan Carlos, 28933 M\'ostoles, Madrid, Spain. Center for Biomedical Technology, Universidad Polit\'ecnica de Madrid, Campus Montegancedo,
28223 Pozuelo de Alarc\'on
Madrid, Spain.}
\email{luis.sola@urjc.es}
%\thanks{Partially supported by the Spanish government project MTM2006-04785.}

%    General info
%\date{}


\begin{abstract}
In this work we deal with vector bundles of rank two on a Fano manifold
$X$ with second and fourth Betti numbers equal to one. We study the nef and pseudoeffective cones of the
corresponding projectivizations and how these cones are related
to the decomposability of the vector bundle. As consequences, we obtain
the complete list of $\P^1$-bundles over $X$ that have a second
$\P^1$-bundle structure, classify all the uniform rank two vector
bundles on this class of Fano manifolds and show the stability of indecomposable Fano bundles (with one exception on $\P^2$).
\end{abstract}

\maketitle
%\tableofcontents
%%%%%%%%%%%%%%%%%%%%%%%%%%%%%%%%%%%%%%%%%%%%%%%%%%

\section{Introduction}\label{sec:intro}
%%%%%%%%

Whilst the classification of vector bundles on the complex projective line is a well stated result, with numberless applications in Algebraic Geometry, the situation in higher dimensions is much more involved, even for vector bundles of low rank. For instance,
no indecomposable rank two vector bundles on $\P^n$, $n\geq 5$
are known so far, despite of the efforts of many algebraic geometers interested on Hartshorne's Conjecture (cf. \cite{Ha}).

On the other hand several results in the literature provide splitting
conditions for vector bundles on $\P^n$. In this direction, the work
of Ancona, Peternell and Wi\'sniewski (cf. \cite{APW}) is of
particular interest for our purposes: they show, using techniques
of Mori theory, that if a rank two vector bundle $\cE$ is indecomposable then $-K_{\P(\cE)}$ cannot
be ample, with the exception of a precise list of bundles.

In a previous paper (cf. \cite{MOS}) we showed that the
amplitude of $-K_{\P(\cE)}$ might be replaced by milder positivity conditions.
Here we address the problem of understanding how splitting may be inferred from positivity, in the more general setting of rank two vector bundles on Fano manifolds of Picard number one.
More concretely, we study the nef and pseudoeffective cones of $\P(\cE)$. They are completely determined by their {\it slopes} $\tau$ and $\rho$ (see Definition \ref{def:threshold}), and we show that their values are tightly related with the (in)decomposability of the bundle (see, for instance, Corollary \ref{cor:tau+rho}).

The structure of the paper is the following: in Sections \ref{sec:threshold} and \ref{sec:posit} we bound the set of values of the pair $(\tau,\rho)$ for an indecomposable vector bundle under different sets of hypotheses. We also include a number of examples in which these invariants are computed and find relations with the Nagata Conjecture for plane curves, see Example \ref{ex:P2}.

In Section \ref{sec:loci} we pay special attention to the behavior of these cones with
respect to rational curves contained in $X$ and with the loci of minimal sections on them:
it is in fact a fundamental question whether these cones are determined by rational curves in $\P(\cE)$.

The second fundamental ingredient of our work, with which we deal in Section \ref{sec:split}, is a splitting criterion based on \cite[Thm.~10.5]{APW} and \cite[Thm.~1]{Ballico}. It basically says, see Corollaries \ref{cor:condition} and \ref{cor:nsapp} for details, that if $\cE$ is indecomposable, then $\rat(\P(\cE))_y$ cannot contain complete curves for the general $y\in\P(\cE)$. As a consequence we obtain that $\cE$ decomposes unless $\tau$ is sufficiently large.

Throughout Sections \ref{sec:posit} and \ref{sec:split} we will also see how our arguments get enhanced if we make some further assumptions on the cones of $\P(\cE)$, semiampleness of the generators of $\mbox{Nef}(X)$, for instance. In this last case (see Proposition \ref{prop:secondray}), applying some results of number theory to our formulas, we obtain severe restrictions on the invariants of $(X,\cE)$.

On the other hand, assuming that $\cE$ is Fano (i.e., $\P(\cE)$ Fano) allows us to put together the Mori machinery and our tools. This leads in Theorem \ref{prop:stFano} to a proof of a Grauert--Schneider type result: indecomposable rank two Fano bundles on Fano manifolds of Picard number one and fourth Betti number $b_4(X)=1$ are stable (with one exception on $\P^2$). Furthermore, under these hypotheses, we give, in Theorem \ref{thm:pbundles}, the complete list of Fano bundles satisfying that $i_X-c_1$ is even, where $i_X$ is the index of $X$ and $c_1$ is an integer corresponding to the first Chern class of $\cE$ as explained in subsection \ref{ssec:setup}.
These bundles are those whose projectivization has a second $\P^1$-bundle structure. This results leads, in particular, to a complete classification of uniform rank two vector bundles
on Fano manifolds under some assumptions (see Corollary \ref{cor:unif}).

The techniques and results of the present paper can be applied also to classify Fano bundles satisfying that $i_X-c_1$ is odd.
%, thus completing the classification  of rank two Fano bundles on Fano manifolds of Picard number one and $H^4(X,\Z) \cong\Z$. 
The proof in this case is longer and more complicated than the one of Theorem \ref{thm:pbundles}, so we will present it in another paper \cite{MOS4}.

Finally in Section \ref{sec:apphart} we show how our techniques may be used to improve some well-known results concerning Hartshorne's Conjecture for codimension two subvarieties of $\P^n$.

\medskip

\noindent{\bf Acknowledgements:} The authors would like to thank E. Ballico for interesting discussions and suggestions, the Korea Institute for Advanced Study (KIAS) for hospitality during part of the preparation of this paper and C. Casagrande for pointing out a mistake in a previous version of the paper. They would also like to thank an anonymous referee for helpful and constructive comments.
%, in special for suggesting us to use  the hypothesis $b_4(X)=1$ instead of $H^4(X,\Z)=\Z$.

\subsection{Setup}\label{ssec:setup}
Throughout this paper we will work in the following general setup:

\begin{set}\label{notn:setup}
$X$ will be a complex Fano manifold of dimension $n$ whose Picard group is generated by the ample line bundle $\cO_X(H_X)$ and $\cE$ will be a
normalized rank two vector bundle on $X$, that is,  whose determinant equals $\cO_X(c_1H_X)$, with $c_1=0$ or $-1$.
\end{set}
Using the isomorphism $\Pic(X)\cong  \Z \langle H_X \rangle $, we will freely identify a line bundle with the corresponding integer. In particular, the determinant of $\cE$ and the anticanonical line bundle are identified with  integers $c_1$ and $i_X$ ({\it index of} $X$). If $a$ is an integer and $\cF$ is a sheaf on $X$, $\cF(a)$ denotes the sheaf $\cF\otimes \cO_X(aH_X)$.

As usual $\P(\cE)$ denotes the Grothendieck projectivization of $\cE$, i.e.
$
\P(\cE)=\Proj_X\left(\bigoplus_{k\geq 0}S^k\cE\right).
$
The tautological line bundle on $\P(\cE)$ will be denoted by $\cO(1)$, the natural projection from $\P(\cE)$ to $X$ by $\pi$, and the pull-back of $H_X$ to $\P(\cE)$ by $H$.
By $L$ we will denote a divisor with associated line bundle $\cO(1)$, and by $-K_{\rel}$ a relative canonical divisor of $\P(\cE)$ over $X$, i.e. a divisor associated to the line bundle $\det(\pi^*\cE^\vee(1))$.
The second Chern class of $\cE$ will be denoted by $c_2(\cE)\in H^4(X,\Z)$, and its discriminant $c_1^2(\cE)-4c_2(\cE)$ by $\Delta(\cE)$.

The {\em Mori cone of} $\P(\cE)$ will be denoted by $\cNE{\P(\cE)}$. It has two extremal rays, $R_1$ and $R_2$, where
$R_1$ corresponds to $\pi$. The ray $R_2$ will often be referred to as
{\em the second extremal ray of} $\P(\cE)$. The dual cone of $\cNE{\P(\cE)}$ is the {\it nef cone of} $\P(\cE)$, denoted by $\Nef(\P(\cE))$. We will also consider the {\em pseudoeffective cone of} $\P(\cE)$, denoted by $\cEff{\P(\cE)}$: the closure of the convex cone generated by
effective divisors. By \cite{BDPP}, it is the dual of the {\em cone of movable curves} $\cME{\P(\cE)}$.

Some of our results will require intersection theory of cycles of codimension one and two. When doing this we will abuse of notation and freely identify cycles with their numerical classes. We will sometimes assume that $b_4(X)=1$, in which case the quotient of $H^4(X,\Z)$ modulo numerical equivalence is isomorphic to $\Z$. We will denote by $\Sigma$ a positive  (meaning that $\Sigma \cdot H^{n-2}>0$) generator $\Sigma$ of this group, so that we
may write $c_2(\cE)=c_2 \Sigma$, $H_X^2=d\, \Sigma$, $\Delta(\cE)=
(dc_1^2 -4c_2) \Sigma =: d\Delta \Sigma$ for some $c_2,d\in\Z$. In particular the usual Chern-Wu relation on  $\P(\cE)$ may be written as
$K_{\rel}^2=\Delta H^2$.

We will sometimes consider the minimum integer $\beta$ such that $\cE(\beta)$ has nonzero global sections, so that there is an exact sequence \begin{equation}\shse{\cO_X(-\beta )}{\cE}{\cI_Z(c_1 +\beta)}\label{eq:shse}
\end{equation}
where $Z\subset X$ has pure codimension two and its cohomology class is $c_2(\cE( \beta))$. The interest of this sequence relies on the fact that, since $X$ is Fano, the bundle $\cE$ decomposes as a sum of line bundles if and only if $Z=\emptyset$. This is in fact equivalent to $Z$ being numerically equivalent to $0$, since $Z \cdot H^{n-2}>0$ when $Z$ is nonempty.

By definition, $\cE$ is {\it stable} (resp. {\it semistable}) if $\beta > - c_1/2$ (resp. $\beta \geq -c_1/2$).
Recall that, by Bogomolov Inequality and Mehta-Ramanathan Theorem,
if $\cE$ is semistable, then $\Delta(\cE)\cdot H_X^{n-2} \le 0$. See \cite{HL} for a complete account on stability of sheaves.
If $X=\P^n$, it is well-known that $\cE$ is not stable if $\Delta=0$, see \cite[Cor. 1]{Bt}. We will prove in Lemma \ref{lem:estdelta=0} that a similar statement holds for Fano manifolds of Picard number one.

Finally, we will consider rational curves in $X$, for which we will adopt the notation and conventions appearing in \cite{K}. Given a rational curve
$\ell\subset X$, the pull-back of $\cE$ via the normalization of $\ell$
takes the form  $\cO_{\P^1}(a)\oplus\cO_{\P^1}(b)$ with $a+b=(H_X\cdot \ell)c_1$.
We will say that $\ell$
has {\em splitting type} $(a,b)$ with respect to $\cE$, or that $\cE$ has splitting type $(a,b)$ with respect to $\ell$.
A rational curve of $H_X$-degree one is called a
{\em line}. Given a family of rational curves $\cM$ in $X$ (i.e. a component of $\rat^n(X)$) of $H_X$-degree $\mu$, with universal family $p:\cU\to\cM$ and evaluation morphism $\ev:\cU\to X$, and given a non negative integer $t$, we will denote by $\cM^t\subset\cM$ the locally closed subset, endowed with the reduced scheme structure, parametrizing curves on which $\cE$ has splitting type $((c_1+t)\mu/2,(c_1-t)\mu/2)$.
In the same way we will use $p^t:\cU^t\to\cM^t$ and $\ev^t:\cU^t\to X$ to denote the corresponding associated maps. We say that $\cE$ is {\it uniform with respect to} $\cM$ if $\cM=\cM^t$ for some $t$.

Given an element $\ell\in\cM^t$ we will sometimes consider minimal sections of $\P(\cE)$ over $\ell$, that will be denoted by $\til\ell$. The set of minimal sections over curves of $\cM^t$ is a family of rational curves in $\P(\cE)$, that we will denote by $\til\cM^t$. We will use $\til p^{\,t}$ and  $\til\ev^t$ to denote the corresponding morphisms.

As usual, a subindex $(\;\,)_x$  on $\cM$, $\cM^t$, $\til\cM^t$, etc., means that we are restricting ourselves to curves passing through the point $x$.

Given a proper family $\cM$ of rational curves or of rational $1$-cycles, and a point $x \in \loc(\cM)$ we will denote by $\cloc_x(\cM)$ the equivalence class of $x$ with respect to the set-theoretic relation associated to the proper proalgebraic relation $\ch(\cU)$, cf. \cite[IV.4.8]{K}. It is the closed subset of $X$ consisting of points that can be joined to $x$ by a connected chain of cycles parametrized by $\cM$.

We will use the fact  that the numerical class of every curve in $\cloc_x(\cM)$ can be written as a linear combination of the numerical classes of irreducible components of cycles parametrized by $\cM$. 

%%%%%%%%%%%%%%%%%%%%%%%%%%%%%%%%%%%%%%%%%%%%%%%%%%

\section{The nef cone of a vector bundle}\label{sec:threshold}
%%%%%%%%

It is well known that the relative anticanonical divisor of a smooth non constant surjective morphism between smooth projective varieties is not ample (cf. \cite[Cor.~2.8]{KMM}). In the next definition we introduce two invariants that give a measure of the (lack of) positivity of $-K_{\rel}$. For simplicity we will stick to Setup \ref{notn:setup}, though the definitions are meaningful in a much broader setting.

\begin{definition}\label{def:threshold}
Given $(X,\cE)$ as in Setup \ref{notn:setup} we denote by $\tau(\cE)$ the only real number such that $-K_{\rel}+\tau H$ is nef but not ample, and we call it the {\it slope of} $\Nef(\P(\cE))$. In a similar manner, we define $\rho(\cE)$ as the only real number such that $-K_{\rel}+\rho H$ is pseudoeffective but not big, and call it the {\it slope of} $\cEff{\P(\cE)}$ (we refer the interested reader to \cite[2.2 B]{L} for details on pseudoeffective divisors).

Equivalently, we may have defined
$$\begin{array}{l}\vspace{0.1cm}\tau(\cE)=\sup\big\{\tau(\cE,C)\,\big|C\mbox{ irreducible curve in }\P(\cE)\big\},\\ \rho(\cE)=\sup\big\{\tau(\cE,C)\,\big|C\mbox{ irred. movable curve in }\P(\cE)\big\},\end{array}$$ where
$$
\tau(\cE,C):=\left\{\begin{array}{ll}\vspace{0.2cm}-\infty&\mbox{ if }C\mbox{ is a fiber of }\pi,\\\dfrac{K_{\rel}\cdot C}{H\cdot C} &\mbox{ otherwise.}\end{array}\right.
$$
If there is no possible confusion we will use $\tau$, $\rho$ and $\tau(C)$ instead of $\tau(\cE)$ and $\tau(\cE,C)$.
\end{definition}

\begin{remark}\label{rem:splittype}
If $\ell$ is a rational curve of splitting type $(a,b)$ with respect to $\cE$ and $\til\ell$ a minimal section of $\P(\cE_{|\ell})$ over $\ell$, then $\tau(\til\ell)=|b-a|/H_X\cdot C\geq 0$. Since it depends only on $\cE_{|\ell}$ and $\ell$, abusing of notation, we will usually write $\tau(\ell)$ instead of $\tau(\til\ell)$. In particular, with the notation introduced in Section \ref{ssec:setup}, we have that $\ell\in\cM^t$ iff $\tau(\ell)=t$.
\end{remark}

Throughout the rest of this section we will discuss some features of $\tau$ and the nef cone of $\P(\cE)$. In the following theorem we show that the lowest value of $\tau$ is only achieved by the trivial bundle. Note that the same proof works for vector bundles of any rank.

\begin{theorem}\label{prop:bound}
Let $(X,\cE)$ be as in Setup \ref{notn:setup}. Then $\tau \geq 0$, and
equality holds if and only if $\cE\cong\cO_X^{\oplus 2}$.
\end{theorem}

\begin{proof}
The first assertion follows by definition of $\tau$ and the fact that the $\Q$-twisted bundle $\cE(-c_1/2)$ is not ample since its determinant is zero.

If $\tau=0$, then
Remark \ref{rem:splittype} tells us that the splitting type of
every rational curve $\ell$ with respect to $\cE$ is either $(0,0)$ or $(-H_X \cdot \ell/2,-H_X \cdot \ell/2)$. Then the conclusion follows from the proposition below (cf. also \cite[Thm. 2.2]{BdS}, where the same result is proved with different techniques).
 \end{proof}

\begin{proposition}\label{prop:trivial}
Let $M$ be a rationally connected manifold and $\cF$ a rank $r$ vector bundle
verifying that $\P(\cF_{|\ell})\cong\ell\times\P^{r-1}$ for every
rational curve in $M$. Then $\P(\cF)\cong M\times\P^{r-1}$ is trivial.
\end{proposition}

\begin{proof}
Let us denote by $\pi:\P(\cF)\to M$ the canonical projection.
Take a component $\cC$ of $\Chow(M)$ containing the class a very free (see \cite[Def. 4.5]{Deb}) rational curve $C\subset M$ and let $\til \cC$ be the component of $\Chow(\P(\cF))$ containing the class of a minimal section $\til C$ of $\P(\cF)$ over $C$.
%Note that $(T_{\P(\cF)|X})_{|\til C}$ is nef, hence
By construction $\til \cC$ dominates $\P(\cF)$ and we may consider the quotient $\varphi:\P(\cF)\dashrightarrow Y$ associated to the proper algebraic relation given by $\til \cC$.

We claim first that $\dim Y\leq r-1$. In fact, take two general points $x_1,x_2\in M$. There exists an irreducible curve $C'$ of $\cC$ joining $x_1$ and $x_2$ hence, given any element $y_1\in\pi^{-1}(x_1)$, we may choose a section of $\P(\cF)$ over $C'$ meeting $y_1$ and $\pi^{-1}(x_2)$. Then $\Cloc(\til \cC)_{y_1}$ intersects the general fiber of $\pi$, hence its dimension is bigger than or equal to $\dim M $. Since $\varphi^{-1}\big(\varphi(y_1)\big)$ contains $\Cloc(\til \cC)_{y_1}$, the claim follows.

Let $L$ and $D$ be divisors in $\P(\cF)$ associated to $\cO_{P(\cF)}(1)$ and $\pi^*(\det\cF)$, respectively, and consider the $\Q$-divisor $L-D/r$ which, by construction, has intersection zero with every cycle of $\til \cC$ and positive intersection with curves on fibers of $\pi$. Furthermore, it is nef on rational curves in $\P(\cE)$. In fact, given a (non-vertical) rational curve $\ell_0\subset \P(\cF)$, we may consider the normalization $\ell$ of its projection into $M$; since $\cF_{|\ell}\cong\cO_{\P^1}(D\cdot\ell_0/r)^{\oplus r}$, it follows that $L-D/r$ is nef on $\P(\cF_{|\ell})$, and thus on $\ell_0$.

In particular $L-D/r$ has intersection zero with every component of a cycle of $\til \cC$. If $\dim Y$ were  strictly smaller than $r-1$, the general fiber of $\varphi$ would intersect  the fibers of $\pi$ in positive dimension, but this is impossible, since $L-D/r$ has positive intersection with curves contained in the fibers of $\pi$. Therefore $\dim Y=r-1$.

Denoting by $Z$ a general fiber of $\varphi$ and by $p:Z\to M$ the natural projection we may consider the projective bundle $\P(p^*\cF)$. It admits a section whose image we denote by $Z'$, and we have a commutative diagram:
$$
\xymatrix{Z'_{} \ar@{^{(}->}[r]&\P(p^*\cF)\ar[r]^{p'}\ar[d]& \P(\cF)\ar[d]^{\pi}\\
&Z^{} \ar@/^1pc/[lu]\ar@{^{(}->}[ru]\ar[r]^p&M}
$$

Since $-K_{\P(\cF)}=-\pi^*K_M +rL-D$ we have that $-K_Z \equiv (-\pi^*K_M)|_{Z}$, hence   $-K_Z = -(\pi|_Z)^*K_M$. This implies that $Z$ is a section of $\pi$, since $M$ is simply connected.

Finally we consider the surjective morphism $\cF\to\cF''$ determined by the section $Z$ and denote by $\cF'$ its kernel. With this notation, the normal bundle $N_Z$ of $Z$ in $\P(\cF)$ is isomorphic to $\cF'^ \vee\otimes\cF''$. On the other hand we know that $N_Z$ is trivial, therefore $\cF'\cong\cF''^ {\oplus r-1}$. Finally, the rational connectedness of $M$ provides $H^1(M,\cO_M)=0$, and the splitting of the sequence
$$
\shse{\cF''^{\oplus r-1}}{\cF}{\cF''}
$$
\end{proof}

As a consequence of Proposition \ref{prop:trivial} we obtain the following result on the stability of $\cE$ when $\Delta=0$, that will be used in Section \ref{sec:split} to make our splitting criteria for vector bundles work also in the case $\Delta=0$ (see Lemma \ref{lem:ns->s}).

\begin{lemma}\label{lem:estdelta=0} Let $(X,\cE)$ be as in Setup \ref{notn:setup} with $\Delta=0$. Then $\cE$ is not semistable unless $\cE$ is trivial.
\end{lemma}

\begin{proof}
Let us assume first that $c_1=0$. By \cite[Thm.~5.1]{MR} no stable vector bundle of rank $\ge 2$ with $c_1=c_2(\cE)\cdot H^{n-2}=0$ on $X$ exists, since $X$ is simply connected. If $\cE$ is semistable then $h^0(\cE)\ne 0$ and $h^0(\cE(-1))=0$, so that $\cE$ is trivial, as $c_2(\cE)$ is numerically trivial.

If $c_1=-1$ and $\cE$ is semistable (hence stable) then $S^2\cE(1)$ is a direct sum of stable vector bundles whose first Chern class is zero, cf. \cite[Thm.~3.2.11]{HL}. Since its rank is three, the second Chern class of every direct summand is zero, too. Applying \cite[Thm.~5.1]{MR} to each summand, we get that $S^2\cE(1)$ is trivial. Hence, for any rational curve $\ell \subset X$, we have $\cE_{|\ell}=\cO(-H_X\cdot \ell/2,-H_X \cdot \ell/2)$, contradicting Proposition \ref{prop:trivial}.
 \end{proof}

We will often consider the restriction of $\cE$ to curves in $X$. Let us then recall some well known facts about ruled surfaces (we refer the reader to \cite[V Sect. 2]{Ha2} for details).

\begin{remark}\label{rem:ruled}
Given a smooth curve $C$ and a $\P^1$-bundle $\P(\cF)$, the vector space $N^1(\P(\cF))$ is
generated by the class of a minimal section $\til C$ and the class of a fiber $F$. Recall that $\cF$ is stable (resp. semistable) if and only if the self-intersection $-e$ of $\til C$ is positive (resp. nonnegative).
If $\cF$ is not semistable, then the nef cone of $\P(\cF)$ is generated by the classes of $F$ and $\til C+eF$ and the pseudoeffective cone by the classes of $F$ and $\til C$; moreover the only irreducible curve whose class does not lie in the nef cone is $\til C$. If $\cF$ is semistable, then the two cones coincide, and they are generated by the classes of $F$ and $2\til C+eF$.

Given any curve in $X$ we may consider the restriction of $\cE$ to its normalization $C$, that we denote by $\cE_{|C}$. Given any irreducible curve $D$ in $\P(\cE)$ we may consider the normalization $\iota:C\to X$ of $\pi(D)$ and the strict transform $D'$ of $D$ in $\P(\cE_{|C})$, obtaining:
$$
\tau(D)=\dfrac{K_{\rel}\cdot D}{H\cdot D}= \dfrac{K_{\P(\cE_{|C})|C}\cdot D'}{\iota^*H\cdot D'}\leq \tau(C_1),
$$
where $C_1$ is an effective $1$-cycle defined as the pushforward into $\P(\cE)$ of $\til C$ if $\cE_{|C}$ is not semistable, and of $2\til C+eF$ otherwise. Note that $K_{\P(\cE_{|C})|C}\cdot {\til C}=e$ and $K_{\P(\cE_{|C})|C}\cdot (2\til C+eF)=0$, hence $\tau(C_1)$ is bigger than $0$ if $\cE_{|C}$ is not semistable and $0$ otherwise.
\end{remark}

In particular, we may state the following

\begin{lemma}\label{lem:tnotss}
Let $(X,\cE)$ be as in Setup \ref{notn:setup}, and assume that $\cE$ is not trivial. Then
$$
\tau=\sup\left\{\left. \tau(\til C)\right|\,\cE_{|C}\mbox{ not semistable}\right\}.
$$
\end{lemma}

\begin{proof}
By the previous discussion, we only need to show that there exists a curve in $X$ on which $\cE$ is not semistable. If this was not the case, then $\tau(\ell)$ would be zero for every rational curve $\ell$ in $X$ and $\cE$ would be trivial by Proposition \ref{prop:trivial}.
 \end{proof}
\begin{remark}\label{rem:tau0}{\rm
A central question on the theory of vector bundles on Fano manifolds is to which
extent the geometry of the bundle is determined by its behavior on rational curves. Proposition \ref{prop:trivial} and Lemma \ref{lem:tnotss} motivate us to define
$$
\tau_0:=\sup\left\{\left. \tau( C)\right|\,C \mbox{ rational}\right\}.
$$ So far, we do not know of any example of rank two vector bundle for which $\tau\neq\tau_0$.}
\end{remark}

%%%%%%%%%%%%%%%%%%%%%%%%%%%%%%%%%%%%%%%%%%%%%%%%%%%%%%%%%%%%%

\begin{remark}\label{rem:tproperties}
It is well known that a vector bundle with $\tau<i_X$ (known in the literature as a {\it Fano bundle}, cf. \cite{APW}) satisfies the following properties:
\begin{itemize}
\item ({\it Rationality}) $\tau\in\Q$.
\item ({\it Base Point Freeness}) $-K_{\rel}+\tau H$ is semiample.
\item ({\it Rational curves in the second ray}) There exists a rational curve $\ell$ for which $\ell \cdot ( -K_{\rel}+\tau H)=0$. In particular $\tau=\tau_0$.
\end{itemize}
It is then natural to ask whether they are satisfied in broader classes of vector bundles on Fano varieties. Section \ref{ssec:ex} below contains a number of examples in which we deal with this question.
\end{remark}

\subsection{Examples}\label{ssec:ex}

\begin{example}\label{ex:split}
If $\cE\cong\cO_X(a)\oplus\cO_X(b)$, then $-K_{\rel}+|b-a|H$ is semiample but not ample, and the corresponding morphism contracts a section of $\P(\cE)$ over $X$, containing the minimal sections of $\cE$ over rational curves in $X$. Note that this in particular shows that $\tau$ is not upper bounded in the class of vector bundles on $X$.
\end{example}

\begin{example}
Indecomposable Fano bundles on projective spaces and quadrics have been extensively studied and completely classified. We refer the interested reader to \cite{SW1}, \cite{SW2} and \cite{APW}.
It is well known that the second contraction of the corresponding projectivization is:
\begin{itemize}
\item  either a smooth blow-up (for $X=\P^2$ or $\Q^4$), or
\item a conic bundle (for $X=\P^2$ or $\Q^3$), or
\item a $\P^1$-bundle (for $X=\P^2,\P^3,\Q^3$ or $\Q^5$).
\item a $\P^2$-bundle (for $\Q^4$).
\end{itemize}
A straightforward computation for every example of the list shows that $\tau=i_X-1$ unless for $\P^d$ bundles, for which $\tau$ equals $i_X- (d+1)$.
\end{example}

There are examples of non Fano bundles for which $\P(\cE)$ enjoys the properties stated in Remark \ref{rem:tproperties}. That is the case, for instance, of the Horrocks--Mumford bundle.

\begin{example}\label{ex:HM}
Let $\cF_{HM}$ denote the Horrocks--Mumford bundle on $\P^4$. The possible splitting
types of lines with respects to $\cF_{HM}$ are $(2,3)$ $(1,4)$, $(0,5)$ and
$(-1,6)$ (cf. \cite{HM}), hence $7\leq\tau_0\leq\tau$. On the other hand the bundle $\cF_{HM}(1)$ is globally generated
(cf. \cite[Prop. 5]{Su}), therefore
$\tau=7$ and the second ray of the Mori cone is generated by the class of a
section of $\P(\cF_{HM})$ over a line of splitting type $(-1,6)$.
\end{example}

Finally we include here an example on $\P^2$ due to Schwarzenberger (see \cite[Thm.~2.2.5]{OSS} for details). We will see that one should not expect $\tau$ to be rational in general. Furthermore, even if $\tau$ is rational, there could be no rational classes in the second extremal ray of $\cNE{\P(\cE)}$.

\begin{example}\label{ex:P2}\cite[Thm.~2.2.5]{OSS} Consider a finite set $\cP=\{P_1,\dots,P_{k}\}\subset\P^2$ of points in the complex projective plane. Blow-up $\P^2$ along $\cP$ to get $\sigma:\B \to \P^2$ and denote by $E=E_1+\dots+E_k$ the exceptional divisor. The {\it Schwarzenberger bundle} maybe defined as the only bundle $\cE$ whose pull-back to $\B$ is an extension of $\cO_\B(-E)$ by $\cO_\B(E)$, whose restriction to $E$ is the standard Euler sequence on $E$.

Now observe that the nefness of $\cE\left(\tau/2\right)$ is equivalent to that of
its pull-back via $\sigma$. Then
$\cE(\tau/2)$ is nef if and only if $\sigma^*H-(2/\tau)E$ is nef and, in particular,
$$\frac{2}{\tau}=\epsilon\left(\cO_{\P^2}(1);\cP\right),$$
where $\epsilon\left(\cO_{\P^2}(1);\cP\right)$ denotes the
{\it Seshadri constant} of the line bundle $\cO_{\P^2}(1)$ with respect to the set of $k$ points  $\cP$. In particular, if $k\geq 9$ one gets
$\tau\geq 2k^{1/2}$ (see, for instance \cite[Sect.~ 8]{Ba}), and the famous Nagata conjecture may be rephrased as follows:

\begin{conjecture}[Nagata]\label{conj:nagata}
With the same notation as above, if $\cP\subset\P^2$ is very general and $k\geq 9$, then $\tau= 2k^{1/2}$.
\end{conjecture}		

A proof of this would provide an example in which $\tau$ is not rational but
this conjecture has been proven only when $k^{1/2}$ is an integer (cf. \cite{Na}).
Note that from Nagata's proof it follows that $\cE(k^{1/2})$ is ample on every
curve in $\P^2$ if $k^{1/2}$ is an integer bigger than or equal to $4$. In particular,
in this case the exact value of $\tau$ is not achieved on any particular
curve in $\P^2$ and $-K_{\rel}+\tau H$ is not semiample.

Let us focus now on the situation in which $\cP$ consists of $9$ very general points.
In this case $\tau=6$ and $\cE(3)$
is nef and not ample on the (unique) smooth
elliptic cubic curve $C$ containing $\cP$.
It is well known that there are no rational curves in $\B$ numerically proportional to the strict transform of $C$, hence $\tau(\ell)\neq 6$ for every rational curve $\ell$.

On the other hand, let $\psi:\P^2\dashrightarrow\P^2$ be the composition of three consecutive Cremona transformations, based on $P_1,P_2,P_3$, then on $P_4,P_5,P_6$ and finally on $P_7,P_8,P_9$. Starting with the
line $\ell_0$ by $P_8$ and $P_9$, consider the sequence of rational curves defined by $\ell_{n+1}:=\psi_*\ell_n$. A direct computation provides $\tau(\ell_n)=6-16/((-1)^n+18n^2+24n+7)$.
Hence we get an example in which $\tau=\tau_0$ but is not computed by a rational curve.
\end{example}

%%%%%%%%%%%%%%%%%%%%%%%%%%%%%%%%%%%%%%%%%%%%%%%%%%

\section{Stability of $\cE$ vs. pseudoeffective divisors of $\P(\cE)$}\label{sec:loci}
%%%%%%%%

In this section we study the relation between the slope $\rho$ of $\cEff{\P(\cE)}$ and the stability of $\cE$. We include some preliminary results on families of minimal sections of $\P(\cE)$ over rational curves that will be useful here and in the forthcoming sections. Recall that:
$$
\rho=\sup\big\{\tau(D)\,\big|D\mbox{ irred. movable curve in }\P(\cE)\big\}.
$$

Given a curve in $X$ with normalization $\iota: C\to X$ such that $\cE_{|C}=\iota^*\cE$ is not semistable, we will denote by $\til C$ the image into $\P(\cE)$ of the minimal section of $\P(\cE_{|C})$ over $C$. Then, arguing as in Remark \ref{rem:ruled}, we may prove the following:

\begin{lemma}\label{lem:alphabetass}
Let $(X,\cE)$ be as in Setup \ref{notn:setup} and assume that $\cE$ is semistable. Then
$$
\rho=\sup\left(\{0\}\cup\left\{\tau(\til C)|\;\til C\mbox{ movable in $\P(\cE)$, }\cE_{|C}\mbox{ not semistable}\right\}\right)\geq 0.
$$
\end{lemma}

\begin{proof}
We begin by showing that $\rho\geq 0$; in fact given $\delta<0$, we will find a movable curve $C_\delta\subset\P(\cE)$ satisfying that $\tau(C_\delta)\geq \delta$. In order to see this, fix a general complete intersection curve $C$ in $X$  and consider the restriction $\cE_{|C}$, which is semistable by Mehta-Ramanathan Restriction Theorem. The nef cone of $\P(\cE_{|C})$, which is equal to $\cME{\P(\cE_{|C})}$, is generated by the fiber $F$ and by $-K_{\rel}$, hence there exists a very ample curve $C_{\delta}$ in $\P(\cE_{|C})$ satisfying $\tau(C_\delta)\geq \delta$. Since we may assume that $C_{\delta}$ is smooth and, by construction, the normal bundle of $C_{\delta}$ in $\P(\cE)$ is globally generated, then $C_{\delta}$ is movable in $\P(\cE)$.

On the other hand, given a movable curve $D$ in $\P(\cE)$ with $\tau(D)>0$, we will show that $D$ is of the form $\til C$, for some $C$. In fact, arguing as in Remark \ref{rem:ruled}, we consider the corresponding ruled surface $\P(\cE_{|C})$, where $C$ is the normalization of $\pi(D)$. If $\cE_{|C}$ is semistable, then $\tau(D)\leq 0$, a contradiction. If $\cE_{|C}$ is not semistable, then we consider the strict transform $D'$ of $D$ in $\P(\cE_{|C})$. If $D'$ is movable in $\P(\cE_{|C})$, then $\tau(D)=\tau(D')<0$, hence we may assume that $D'$ is not movable, therefore $D'=\til C$.
 \end{proof}

We will now make a closer analysis of the positivity of divisors in $\P(\cE)$ with respect to rational curves in $\P(\cE)$. We begin with the following straightforward lemma (see, for instance, \cite[V Sect. 2]{Ha2}).

\begin{lemma}\label{lem:splitsection}
Let $(X,\cE)$ be as in Setup \ref{notn:setup} and let $\cM$ be a dominating
family of rational curves of $H_X$-degree $\mu$ in $X$. Let $D$
be an effective divisor numerically proportional to $-K_{\rel}+bH$, and $\ell$ be a curve in $\cM^t$, $t\geq 0$. Then:
\begin{itemize}
\item If $b<-t$, then $\P(\cE_{|\ell})\subset D$.
\item If $t> b\geq -t$ and $\P(\cE_{|\ell})$ is not
contained in $D$, then $\P(\cE_{|\ell})\cap D$ contains the (unique) minimal section $\til{\ell}$, being exactly the section when $b=-t$.
\end{itemize}
In particular, if $\cM^t$ dominates $X$ and $b<t$, then every linear system of the form
$|k(-K_{\rel}+bH)|$ contains a fixed component $F$, where $F$ denotes the
closure of $\loc(\til\cM^{t})$.
\end{lemma}

A weaker version of this lemma provides a splitting criterion for uniform
bundles:

\begin{corollary}\label{cor:splitsection}
Let $(X,\cE)$ be as in Setup \ref{notn:setup}. Assume moreover that $\cM=\cM^t$ is a uniform unsplit covering
family of rational curves on $X$, with $t>0$, and that
$H^0(\cE(-(c_1+t)/2))\ne 0$. Then $\cE$ is decomposable.
%=\cO_X((c_1+t)/2)\oplus \cO_X((c_1-t)/2)$.
\end{corollary}

\begin{proof}
Take $D \in |L-((c_1+t)/2)H|=|1/2(-K_{\rel}+tH)|$ and denote by $Z\subset X$ the biggest subset such that $\pi^{-1}(Z)
\subset D$, whose cohomology class is $c_2(\cE((c_1+t)/2))$. Hence it suffices
to show that $Z$ is empty: assume the contrary.
If $\ell\in\cM$ is not contained in $Z$ then any nonzero
section of $\cE((c_1+t)/2))$ does not vanish identically on $\ell$, hence
is nowhere vanishing on  $\ell$ since $\cE$ has splitting type $((c_1+t)/2,(c_1-t)/2)$, therefore $\ell \cap Z=\emptyset$.
It follows that $\cloc_z(\cM) \subset Z$, for every $z \in Z$ and this contradicts $\Pic(X)\simeq \Z$ and $\cM$ unsplit.
 \end{proof}

\begin{remark}\label{rem:splitsectionAW}Let us observe that the same result holds in the case $t=0$  without assuming $H^0(\cE(-c_1/2))\neq 0$ (cf. \cite[Prop.~1.2]{AW}).
\end{remark}

In Lemma \ref{lem:alphabetass} we have seen that $\rho\geq 0$ when $\cE$ is semistable. The next proposition shows that the converse is also true and that, moreover, if $\cE$ is not semistable then the pseudoeffective cone of $\P(\cE)$ is completely determined by the maximal destabilizing subsheaf of $\cE$.

\begin{proposition}\label{prop:alpha=beta}
Let $(X,\cE)$ be as in Setup \ref{notn:setup}, let $\beta$ be the minimum integer  such that $\cE(\beta)$ has nonzero global sections  and let $|k(-K_{\rel}+bH)|$ be a nonempty
linear system in $\P(\cE)$. Either $b\geq 0$ or $|k(-K_{\rel}+bH)|$ has a base component numerically proportional to $-K_{\rel}+\rho H$. In particular, if $\cE$ is not semistable then $\rho=2\beta+c_1<0$.
\end{proposition}

\begin{proof}
Let $\cM$ be a dominating family of rational curves in $X$, and denote by $\mu$ its $H_X$-degree.
Denote $t:=\tau(\ell)$, $\ell\in\cM$ general, and let $F$ be the closure of $\loc(\til\cM^{t})$.

Assume $b<0$; then $b<0\leq t$ and Lemma \ref{lem:splitsection} above tells us that $|k(-K_{\rel}+bH)|$ has a base component $F$, the
closure of $\loc(\til\cM^{t})$. Moreover, either $F$ or $k(-K_{\rel}+bH)-F$ are of type $j(-K_{\rel}+cH)$ with $\rho\leq c<0$.

If $F$ is not of that type, then we may apply the argument above to $|k(-K_{\rel}+bH)-F|$, obtaining that it has a base component $F$. After a finite number of steps, the nonempty linear system $|k'(-K_{\rel}+b'H)|:=|k(-K_{\rel}+bH)-rF|$  will not contain $F$ as a base component, contradicting that $b'$ will still be smaller than $0$.

We may then assume that $F\equiv j(-K_{\rel}+cH)$ with $\rho\leq c<0$. Arguing in a similar way, we get that
every multiple $rF$ of $F$  has $F$ in its base locus, hence it follows that
$|rF|$ is zero-dimensional. In particular $F$
is not big, hence $F$ is numerically proportional to $(-K_{\rel}+\rho H)$. This proves the first part of the statement.

For the second part, assume that $\cE$ is not semistable. Then $\beta<-c_1/2$ by definition, and we may apply the claim above to $|1/2(-K_{\rel}+(2\beta+c_1)H)|=|L+\beta H|\neq\emptyset$ to conclude that it has a base component $F$. But $|L+\beta H|$ consists of irreducible unisecant - i.e. of relative degree $1$ - divisors, hence this is possible only when $F\in|L+\beta H|$. It follows that $\rho= 2\beta+c_1$.
 \end{proof}

\begin{remark}\label{rem:alphanotinteger}
The equality $\rho= 2\beta+c_1$ holds also for bundles which are semistable, but not stable (and so have $\beta=c_1=0$). In fact $\rho \ge 0$ by Lemma \ref{lem:alphabetass} but, on the other hand $|-K_{\rel} |=| 2L | $ is not empty. This is no longer true for stable bundles:
the blow-up $Y$ of $\P^3$ along a twisted cubic (see \cite{SW1}) is isomorphic to a $\P^1$-bundle $\P(\cE)$ over $\P^2$. The exceptional divisor of $Y$ over $\P^3$ is not unisecant, hence $\rho\neq 2\beta+c_1$.
\end{remark}

The rest of the section is devoted to a Grauert--M\"ulich type result for Fano manifolds.
The classical Grauert--M\"ulich theorem (cf. \cite[II, Sect. 2]{OSS}) tells us that the general splitting type of a vector bundle on $\P^n$ with respect to a line cannot have gaps of length bigger than one. For other base varieties, one may still control the gaps between the slopes of the Harder-Narasimhan filtration of the restriction of a vector bundle on a manifold to general complete intersection curves (see, for instance \cite{FHS}, \cite{MRs} and \cite{Fl}). However it was already noted by Hirschowitz in \cite{Hi} that the standard arguments work in a much broader setting.

In the case of dominating families of rational curves of Fano manifolds, a similar result can be obtained as a by-product of the arguments of this section; we state it here for bundles of rank two. We will make use of the following Remark.

\begin{remark}\label{lem:a=bprelim}
Let $(X,\cE)$ be as in Setup \ref{notn:setup}, $\cM$ be a dominating family of rational curves, satisfying that $\cM_x$ is irreducible for $x\in X$ general. Let $t:=\tau(\ell)$,  $\ell\in\cM$ general, and assume $t\neq 0$. Then the closure of $\loc(\til\cM^{t})$ is $\P(\cE)$ or a unisecant divisor.
\end{remark}

\begin{proof}
Assume that $\til\cM^t$ does not dominate $\P(\cE)$ and let $D$ be the closure of $\loc(\til\cM^{t})$.
Let $f$ be a general fiber of $\pi$; by the generality of $f$ we can assume that $\pi(f) \not \in \pi(D \setminus \loc(\til\cM^{t}))$, and that the intersection of $D$ and $f$ is transversal;  by the irreducibility of $\cM_x$ the intersection number is then one.
%Denoting by $p:\cU\to\cM$ the universal family parametrized by $\cM$, with evaluation morphism $\ev$ onto $X$, we have an open set $\cU_0\subset\cU$ such that $p(\cU_0)\subset\cM^t$ and the curves $p(u)$ are smooth at $\ev(u)$.
%
%Then there is a morphism $\f:\cU_0\to\P(\cE)$, sending every curve to the minimal section of $\P(\cE)$ over it. Note that this minimal section is unique since $t>0$.
%Denote by $p_0$ and $\ev_0$ the restrictions of $p$ and $\ev$ to $\cU_0$.
%By construction $\f(p_0^{-1}(\ell))\cap\pi^{-1}(x)$ is one point for every $x$ in the
%image. Denote by $F_0$ the image of $\cU_0$ via $\f$.
%
%Since by hypothesis
%$\cM_x$ is irreducible for the general $x$, then $F_0\cap\pi^{-1}(x)=\f(\ev_0^{-1}(x))$ is
%irreducible. Since $F_0$ is contained in $F$ by construction, that intersection is exactly one point. But then, shrinking $F_0$, we may assume that
%it is isomorphic, via $\pi$, to an open set of $X$. Then
%its closure $F$ must be unisecant.
 \end{proof}

The following result may be interpreted as a Grauert-M\"ulich type theorem for rational
curves on Fano manifolds.

\begin{proposition}\label{prop:section} Let $(X,\cE)$ be as in Setup \ref{notn:setup}. Let $\cM$ be a dominating family of rational curves on $X$ such that $\cM_x$ is irreducible for the general $x\in X$. Let $\mu$ denote the $H_X$-degree of $\cM$. Either $\tau(\ell)\leq 1/\mu$ for the general $\ell\in\cM$, or $\cE$ is destabilized by a line bundle of degree
$\beta=-\big(\tau(\ell)+c_1\big)/2$.
\end{proposition}

\begin{proof} Set $t:=\tau(\ell)$,  $\ell\in\cM$ general, assume $t\geq 2/\mu$ and consider the corresponding families $\cM^t$ and $\til\cM^t$. We will prove that the evaluation morphism $\til\ev^t:\til\cU^t\cong\cU^t\to\P(\cE)$ is not dominant. Then, by Remark \ref{lem:a=bprelim},  the closure of $\loc(\til\cM^{t})$ will be a unisecant divisor, corresponding to an element
in
$\Hom_X(\cO_X(-b),\cE)$
for some $b$. Restricting to $\ell$ we see that $-b=
\big(\tau(\ell)+c_1\big)/2$. Moreover, since $\cM^t$ covers $X$, $H^0(X,\cE(b'))=0$ for $b'<b$. This concludes the proof.

In order to see that $\til\ev^t$ is not dominant, note that for the general
element $\ell$ of $\cM^t$, the restriction of the tangent bundle of $X$ to $\ell$ is nef. Since $d(\ev^t)_{|\ell}:(T_{\cU^t})_{\ell}\to (T_X)_{|\ell}$ is the evaluation of global sections of $(T_X)_{|\ell}$, it follows that its kernel satisfies
$(T_{\cU^t|X})_{|\ell}\cong\cO(-1)\oplus\dots\oplus\cO(-1)$.
On the other hand
$(T_{\P(\cE)|X})_{|\ell}\cong\cO(t\mu)$,
with $t\mu\geq 2$  by assumption, hence there are no nonzero morphisms from
$(T_{\cU^t|X})_{|\ell}$ to $(T_{\P(\cE)|X})_{|\ell}$ and $d(\til\ev^t)_{|\ell}$ is not generically surjective.
 \end{proof} 

%%%%%%%%%%%%%%%%%%%%%%%%%%%%%%%%%%%%%%%%%%%%%%%%%%

\section{The pseudoeffective cone of $\P(\cE)$}\label{sec:posit}
%%%%%%%%

In this section we will explore the relation between $\tau(\cE)$ and $\rho(\cE)$. Unless otherwise stated we will always assume the following:

\begin{set}\label{ass:sec4}
 $(X,\cE)$ will be as in Setup \ref{notn:setup} with $b_4(X)=1$. As described in \ref{ssec:setup} we will use the following equalities of numerical classes of cycles:
$c_2(\cE)=c_2 \Sigma$ and $H_X^2=d\, \Sigma$ with $c_2,d \in \Z$, $d>0$; $\Delta(\cE)=
(dc_1^2 -4c_2) \Sigma =: d\Delta \Sigma$ and $K_{\rel}^2=\Delta H^2$.
\end{set}

The main idea we will use is the following:
\begin{remark}\label{rem:idea3}
Pseudoeffective divisors have non negative intersection with movable classes of $1$-cycles, for instance with complete intersections of nef divisors. Thus we have the following restrictions:
\begin{equation}\label{eq:posit2}
(-K_{\rel}+\rho' H)\cdot(-K_{\rel}+\tau H)^{j}\cdot H^{n-j}\geq 0, \mbox{ for all }j\in\{0,\dots,n\}, \;\rho'\geq\rho.
\end{equation}

By the Chern-Wu relation $K_{\rel}^2=\Delta H^2$, the equality
$H^{n+1}=0$ and the inequality $-K_{\rel}H^n>0$, inequality (\ref{eq:posit2}) reduces, for each $j$, to
\begin{equation}\label{eq:aj}
\dfrac{(\rho'+\sqrt{\Delta})(\tau+\sqrt{\Delta})^j-(\rho'-\sqrt{\Delta})(\tau-\sqrt{\Delta})^j}{\sqrt{\Delta}}\geq 0,\;j\in\{0,\dots,n\}, \;\rho'\geq\rho
\end{equation}
if $\Delta\neq 0$, and to
$$
\tau^{j-1}(\tau+j\rho')\geq 0, \mbox{ for all }j\in\{0,\dots,n\}, \;\rho'\geq\rho
$$
if $\Delta=0$. In this case, applying Theorem \ref{prop:bound}, Lemma \ref{lem:estdelta=0} and Proposition \ref{prop:alpha=beta}, all the inequalities reduce to
\begin{equation}\label{eq:posit0n}
\tau+n\rho\geq 0.
\end{equation}
\end{remark}

\subsection{\bf The case $\Delta<0$.}\label{ssec:3<0}
Denoting by $\arg(z)\in\big[0,2\pi\big)$ the argument of a complex number $z\neq 0$, (\ref{eq:aj}) is, in this case, equivalent to:
\begin{equation}\label{eq:posit21}
\arg\big((\rho'+i\sqrt{-\Delta})(\tau+i\sqrt{-\Delta})^{j}\big)\leq\pi,\mbox{ for all }j\in\{0,\dots,n\}.
\end{equation}
Note that the inequality for $j=n$ implies the rest, then (\ref{eq:aj}) is equivalent to
\begin{equation}\label{eq:posit22}
\arg\big((\rho'+i\sqrt{-\Delta})(\tau+i\sqrt{-\Delta})^{n}\big)\leq\pi.
\end{equation}

Applying (\ref{eq:posit22}) above to $\rho'=\tau$, we get the following inequality, which can be read  as the positivity of certain Schur polynomials of the minimal nef twist of $\cE$, and that can be interpreted as a lower bound for $\tau$ in terms of the discriminant $\Delta$ and the invariants of $X$.

\begin{lemma}\label{lem:posit}
Let $(X,\cE)$ be as in Setup \ref{ass:sec4}. Then:
$$
c_1^2<\dfrac{4c_2}{d}\leq c_1^2+\tau^2\tan^2\left(\dfrac{\pi}{n+1}\right).
$$
\end{lemma}

We obtain stronger restrictions by applying equation (\ref{eq:posit22}) to $\rho'=\rho$. However, it is not clear whether $\rho<\tau$. We already know this occurs when $\cE$ is not semistable (see Theorem \ref{prop:bound} and Proposition \ref{prop:alpha=beta}) The next proposition shows that this is also the case if $\tau$ is rational and $n\neq 2,3,5$.

\begin{proposition}\label{prop:notbig}
Let $(X,\cE)$ be as in Setup \ref{ass:sec4}, $\cE$ non trivial. If $\tau\in\Q$, then $-K_{\rel}+\tau H$ is big unless $\Delta<0$ and $n=2,3$ and $5$. In particular $\rho<\tau$.
\end{proposition}

\begin{proof}
Assume that $-K_{\rel}+\tau H$ is not big. On one hand we get $\rho=\tau$, hence $\rho \ge 0$ by Theorem \ref{prop:bound}, so $\cE$ is semistable by Proposition \ref{prop:alpha=beta} and $\Delta\leq 0$. On the other we obtain an equality $(-K_{\rel}+\tau H)^{n+1}=0$ that, arguing as above, leads us to $\tau=0$, and to the triviality of $\cE$ by Theorem \ref{prop:bound}, if $\Delta=0$.

If $\Delta<0$ then we get:
$$
\tan^2\left(\dfrac{\pi}{n+1}\right)=\frac{-\Delta}{\tau^2}\in\Q.
$$
The algebraic degree of $\tan\left(\pi/(n+1)\right)$ over $\Q$ is classically known (see \cite[pp. 33-41]{Ni}, see also \cite[Prop.~2]{Ca}), and  one may check directly that the only possible values of $n$ are $2,3$ or $5$.
 \end{proof}

\subsection{\bf The case $\Delta \ge 0$.}\label{ssec:3>0}

In this case $\cE$ is not semistable, unless $\cE \simeq \cO_X \oplus \cO_X$ (cf. Lemma \ref{lem:estdelta=0}),
hence, by Proposition \ref{prop:alpha=beta} and Remark \ref{rem:alphanotinteger}, we know that $\rho=  2\beta+c_1 \le 0$
and $c_2(\cE( \beta))\geq 0$. In particular, since the numerical classes $c_2(\cE(k))=(c_2+dkc_1+dk^2)\Sigma$ and $c_2(\cE(-c_1-k))$ are equal, we get:
\begin{equation}\label{eq:m+alpha}
\rho=  2\beta+c_1\leq-\sqrt{\Delta}
\end{equation}

\begin{remark}\label{rem:extremedelta2}
Equality in (\ref{eq:m+alpha}) holds if and only if $\cE$ is a direct sum of line bundles.
In fact if equality holds then $c_2(\cE( \beta))=0$ and $\cE$ splits.
The converse follows from a direct computation.
\end{remark}

On the other hand, the set of restrictions (\ref{eq:aj}) provide the following:

\begin{lemma}\label{lem:posit4}
Let $(X,\cE)$ be as in Setup \ref{ass:sec4} with $\Delta \ge 0$. Then
\begin{eqnarray}\vspace{0.2cm}
&& \tau\geq\sqrt{\Delta},\quad\mbox{and} \label{eq:6}\\
&& -\sqrt{\Delta}-\varepsilon\leq \rho\leq-\sqrt{\Delta}, \label{eq:7}
\end{eqnarray}
where
$$\varepsilon= \dfrac{2\sqrt{\Delta}\big(\tau-\sqrt{\Delta}\big)^n}
{\big(\tau+\sqrt{\Delta}\big)^n-\big(\tau-\sqrt{\Delta}\big)^n}\quad  \text{if}\;\; \Delta >0 \quad \text{and} \quad
\varepsilon= \dfrac{\tau}{n} \quad\text{if}\;\; \Delta =0
$$
\end{lemma}

\begin{proof} With the same notation as in Remark \ref{rem:idea3}, if $\Delta >0$ inequalities (\ref{eq:aj}) give
$$
(\rho+\sqrt{\Delta})(\tau+\sqrt{\Delta})^j\geq(\rho-\sqrt{\Delta})(\tau-\sqrt{\Delta})^j\mbox{, for all }j\leq n.
$$
By (\ref{eq:m+alpha}), $\rho+\sqrt{\Delta}\leq 0$, therefore $\rho-\sqrt{\Delta}<\rho+\sqrt{\Delta}\leq 0$, and the previous inequality tells us in particular that $\tau\geq\sqrt{\Delta}$: otherwise
$$
(\rho+\sqrt{\Delta})(\tau+\sqrt{\Delta})\leq(\rho+\sqrt{\Delta})(\tau-\sqrt{\Delta})<(\rho-\sqrt{\Delta})(\tau-\sqrt{\Delta}).
$$

Finally, since $\tau-\sqrt{\Delta}\geq 0\geq \rho+\sqrt{\Delta}$, the set of inequalities in (\ref{eq:aj}) reduces to the one given obtained for $j=n$
and a simple computation concludes the proof.

If $\Delta =0$ then (\ref{eq:6}) follows from Theorem \ref{prop:bound} and   (\ref{eq:7}) is a restatement of (\ref{eq:posit0n}).  \end{proof}

\begin{remark}\label{rem:extremedelta}
Equality in (\ref{eq:6}) holds if and only if $\cE$ is a direct sum of line bundles.
In fact if $\tau=\sqrt{\Delta} >0$, then $\varepsilon=0$, equality holds in (\ref{eq:m+alpha}) and $\cE$ splits by Remark \ref{rem:extremedelta2}.
If else $\tau=\Delta=0$, the same result follows from Theorem \ref{prop:bound}.
The converse follows from a direct computation.
\end{remark}

Bounding $\varepsilon\leq\varepsilon':=(\tau-\sqrt{\Delta})/n$ we obtain the following inequality:

\begin{corollary}\label{cor:boundeps}
Let $(X,\cE)$ be as in Setup \ref{ass:sec4} with $\Delta\geq 0$. Then:
\begin{equation}
\tau+n\rho+(n-1)\sqrt{\Delta}\geq 0
\end{equation}
\end{corollary}

Finally we obtain the following two splitting criteria

\begin{corollary}\label{cor:posit4}
Let $(X,\cE)$ be as in Setup \ref{ass:sec4} with $\Delta \ge 0$, and consider the following interval:
$$I=\left[-\dfrac{c_1+\sqrt{\Delta}}{2}-\dfrac{\varepsilon'}{2},-\dfrac{c_1+\sqrt{\Delta}}{2}\right].$$
Then:
\begin{itemize}
 \item $I\cap\Z\neq\emptyset$, and
 \item if $I\cap\Z=-(c_1+\sqrt{\Delta})/2$, then $\cE$ splits as a direct sum of line bundles.
\end{itemize}
\end{corollary}

\begin{corollary}\label{cor:numsplit}
Let $(X,\cE)$ be as in Setup \ref{ass:sec4} with $\Delta \geq 0$. If $-(c_1+\sqrt{\Delta})/2 \in \Z$ and
\begin{equation*}\label{eq:numsplit2}
\tau< 2n+\sqrt{\Delta},
\end{equation*}
then $\cE$ is decomposable.
\end{corollary}

\begin{proof}
It is immediate to see that $\sqrt{\Delta}$ is an integer if and only if $-(c_1+\sqrt{\Delta})/2\in\Z$; then by Corollary \ref{cor:posit4} it suffices to check that $\varepsilon'<2$.
 \end{proof}

We finish this section discussing the effects on $\cEff{\P(\cE)}$ of assuming that $-K_{\rel}+\tau H$ is semiample. More concretely, we will prove that if the associate contraction is of fiber type or divisorial, then $\rho$ is completely determined from the rest of the invariants of $(X,\cE)$. We will make use of the following observation on the set of inequalities (\ref{eq:posit2}), whose proof follows from a straightforward analysis of the different forms that (\ref{eq:posit2}) takes according to the sign of $\Delta$.

\begin{lemma}\label{lem:descend}
Let $(X,\cE)$ be as in Setup \ref{ass:sec4} and assume that $\cE$ is indecomposable. Then
$$
(-K_{\rel}+\rho' H)\cdot(-K_{\rel}+\tau H)^{j}\cdot H^{n-j}>0,
$$
for every $\rho' \geq \rho$ unless $(j,\rho')=(n,\rho)$.
\end{lemma}

\begin{proposition}\label{prop:secondray}
Let $(X,\cE)$ be as in Setup \ref{ass:sec4}, with $\cE$ indecomposable. Assume further that
there exists a morphism $\f:\P(\cE) \to Y$ contracting the second extremal ray $R_2$ of $\P(\cE)$. Then:
\begin{itemize}
\item If $\f$ is of fiber type, then  $\cE$ is stable, $\rho=\tau$ and $n=\dim Y=2,3$ or $5$.
\item If $\f$ is divisorial, then $\dim \f(\Exc(\f))=n-1$ and $\rho$ is determined by $\tau$ and $\Delta$. In particular, if $\Delta <0$  then
\begin{equation}\label{eq:schr2}
\arg\big((\rho+i\sqrt{-\Delta})(\tau+i\sqrt{-\Delta})^{n}\big)=\pi.
\end{equation}
\end{itemize}
\end{proposition}

\begin{proof}

If $\f$ is a fiber type contraction, then $-K_{\rel}+\tau H$ is not big, hence $\rho=\tau$ and, by Theorem \ref{prop:bound}, Proposition \ref{prop:alpha=beta} and Remark \ref{rem:alphanotinteger}, $\cE$ is stable. In particular $\Delta < 0$ and we can apply Proposition \ref{prop:notbig} to get $n=2,3$ or $5$.
Finally, Lemma \ref{lem:descend} provides $(-K_{\rel}+\tau H)^n\cdot H > 0$, which forces $\dim Y= n$.

If $\f$ is a divisorial contraction, then the class of $\Exc(\f)$ is effective and not big, hence it is numerically proportional to $-K_{\rel}+\rho H$. Therefore we have $$(-K_{\rel}+\rho H)\cdot(-K_{\rel}+\tau H)^{n-1}\cdot H>0,\mbox{ and }(-K_{\rel}+\rho H)\cdot(-K_{\rel}+\tau H)^{n}=0.$$
The first expression tells us that $\dim \f(\Exc(\f))\geq n-1$, whereas, arguing as above, the reduction modulo $K_{\rel}^2=\Delta H^2$ of the second shows that $\rho$ is uniquely determined by the values of $\tau$ and $\Delta$. In particular, when $\Delta<0$ the relation between $\rho$, $\Delta$ and $\tau$ is equation (\ref{eq:schr2}).
 \end{proof}

%%%%%%%%%%%%%%%%%%%%%%%%%%%%%%%%%%%%%%%%%%%%%%%%%%

\section{Splitting criteria}\label{sec:split}
%%%%%%%%

Throughout this section we will use the following notation and assumptions unless otherwise stated:

\begin{set}\label{ass:h4=1}
Let $(X,\cE)$ be as in Setup \ref{ass:sec4}. We denote by $\cM$ a dominating family of rational curves on $X$ of $H_X$-degree $\mu$.
\end{set}

\begin{remark}\label{rem:idea}
By definition, the zero locus of a section of $\cE(\beta)$ has pure codimension $2$, hence a general element $\ell_0$ of $\cM$ avoids it (cf. \cite[II, Prop. 3.7]{K}), so that $\cE_{|\ell_0}$ is an extension of $\cO((c_1+\beta)\mu)$ by $\cO(-\beta\mu)$. In particular, if $\cE$ is not stable  then for all $\ell\in\cM$ we have
\begin{equation}\label{eq:trho}
\tau(\ell) \geq\tau(\ell_0)=-(2\beta+c_1)= -\rho,
\end{equation}
where the last equality follows from Proposition \ref{prop:alpha=beta} and Remark \ref{rem:alphanotinteger}
\end{remark}

The following two results are based on \cite[Thm. 10.5]{APW} and
\cite[Thm. 1]{Ballico}.

\begin{lemma}\label{lem:newns}
Let $(X,\cE)$  be as in Setup \ref{ass:h4=1}, assume that, for some rational number $t$
 there is a surface $S \subset \P(\cE)$ such that $\pi_{|S}$ is finite and that
 $(-K_{\rel}+tH)\cdot C=0$ for every curve $C \subset S$. Then $\Delta=t^2$.
\end{lemma}

\begin{proof}
By hypothesis, $(K_{\rel})_{|S}$ and $tH_{|S}$ are numerically equivalent. Hence $((K_{\rel})_{|S})^2=t^2H_{|S}^2$. But $(K_{\rel})^2$ is numerically equivalent to $\Delta H^2$ and $\pi_{|S}$ is finite, therefore $\Delta=t^2$.
 \end{proof}

The numerical conditions of the previous lemma lead to a splitting criterion.

\begin{lemma}\label{lem:ns->s}
Let $(X,\cE)$ and $\cM$ be as in Setup \ref{ass:h4=1}. If $\Delta=t^2$ for some positive rational number $t$ and there exists a curve $\ell\in\cM$ satisfying $\tau(\ell)\leq t$, then $\cE$ is decomposable.
\end{lemma}

\begin{proof}
Since $\Delta \geq 0$, then, by Bogomolov inequality and Lemma \ref{lem:estdelta=0}, $\cE$ is not semistable; the assumption on $\tau(\ell)$ yields, by (\ref{eq:trho}) that $\rho\ge -\sqrt{\Delta}$; combining it with (\ref{eq:m+alpha}) we get $\rho=-\sqrt{\Delta}$, and $\cE$ is decomposable by Remark \ref{rem:extremedelta2}.
 \end{proof}

%It is well known that a vector bundle on $\P^n$ splits a sum of line bundles when its restriction to some plane $\P^2 \subset \P^n$ verifies the same property  (cf. \cite[Thm. 2.3.2]{OSS}); we can generalize this result for rank two bundles on Fano manifolds as follows:
%
%\begin{corollary}\label{cor:splitsurface}
%Let $(X,\cE)$ be as in Setup \ref{ass:h4=1} and assume that there is a surface $S \subset X$ which contains a free rational curve such that $\cE_{|S}$ is decomposable. Then $\cE$ is decomposable.
%\end{corollary}

A more general splitting criterion can be stated as follows:

\begin{corollary}\label{cor:condition}
Let $(X,\cE)$ and $\cM$ be as in Setup \ref{ass:h4=1}. Assume that there exists a rational number $t$ such that
\begin{equation}\label{eq:ns}
 \til\cM_y^{t} \; \text{contains a complete curve $T$ for some}\; y \in \P(\cE).
\end{equation}
Then $\cE$ is decomposable.
\end{corollary}

\begin{proof} Let $S$ be the locus of curves parametrized by the complete curve $T$. Every curve in $S$ is numerically proportional to a curve $\ell$ of $\til\cM_y^{t}$. Since $\tau(\ell)=t$ it follows that $-K_{\rel}+tH$ is numerically trivial on $H$, hence $\Delta=t^2$ by Lemma \ref{lem:newns} and the splitting follows from Lemma \ref{lem:ns->s}.
 \end{proof}

\begin{remark}\label{cor:nsapp}
The same arguments work if we assume that for some $t$, the family $\til\cM^t$ is unsplit and for some $y \in \loc(\til\cM^{t})$ there is a component of $\Cloc_y(\til\cM^{t})$ of dimension $\ge 2$.
\end{remark}

Keeping track of the dimensions of the families parametrizing rational curves of fixed splitting type passing by a general point, one may translate condition (\ref{eq:ns}) into a numerical splitting criterion:

\begin{theorem}\label{prop:nssemistable}
Let $(X,\cE)$ and $\cM$ be as in Setup \ref{ass:h4=1}, and assume that $\cM_x$ is proper for a general $x\in X$ and that $\cE$ is not stable. Assume moreover that
\begin{itemize}
\item  $\tau < 2i_X -\rho-4/\mu$\; if $\cE $ is not semistable, and
\item $\tau < 2i_X-6/\mu$\;  if $\cE$ is semistable.
\end{itemize}
Then $\cE$ is decomposable.
\end{theorem}

\begin{proof} Let $x$ be a general point in $X$ and denote by $t_{\min}$ (resp. $t^{\max}$) the minimum (resp. the maximum) integer such that $\cM_x^{t_{\min}}\neq\emptyset$ (resp. $\cM_x^{t^{\max}}\neq\emptyset$). By Remark \ref{rem:idea}, $t_{\min}=-\rho\geq 0$. Since $\cE$ is not stable, by Proposition \ref{prop:alpha=beta} and Remark \ref{rem:alphanotinteger} we have that $\rho=2\beta+c_1$.
Let $D$ be a divisor in $ |L+\beta H|=|\frac{1}{2}(-K_{\rel}+\rho H)|$, and $t\in[t_{\min},t^{\max}]$. By Lemma \ref{lem:splitsection}
we deduce that $\loc( \til\cM^t) \subseteq D$ unless $t=t_{\min}= \rho=0$.

Assume first that $\cE$ is not semistable. The divisor $D$ meets the fiber over a general $x \in X$ in a point hence, from the argument above, we get that the existence of a complete curve in $\cM^t_x$ for some $t$, implies Condition (\ref{eq:ns}).
Since $\cM_x$ is proper, we may thus express a necessary condition for the bundle not to split as follows (cf. \cite[Cor. 10.6]{APW}):
\begin{enumerate}
\item[a)] $\dim  \cM^{t^{\max}}_x=0$,
\item[b)] $\dim \overline {\cM_x^t} - \dim ( \overline {\cM_x^t} \cap (\cup_{b >t}\cM_x^b )) \le 1$ if  $t_{\min} \le t < t^{\max}$.
\end{enumerate}
In particular $\#\{\cM^t_x\} \ge \dim \cM_x +1=i_X\mu-1$. On the other hand
\begin{equation}\label{eq:ss2}
\#\{\cM^t_x\}\leq \dfrac{\mu}{2}(t^{\max}-t_{\min})+1\leq \dfrac{\mu}{2}(\tau+\rho)+1,
\end{equation}
that combined with the previous equation, gives the first part of the statement.

The case of $\cE$ semistable but not stable (which corresponds to $t_{\min}= \rho=0$) is slightly different: since $\loc(\til\cM^0) \not \subset D$, we only know that if $\cE$ is indecomposable, then $\#\{\cM^t_x\} \ge \dim \cM_x =i_X\mu-2$, and we conclude by combining this with (\ref{eq:ss2}). \end{proof}

\begin{corollary}\label{cor:newbound}
With the same notation as above, assume that $\Delta\geq 0$ and that $\tau < 2i_{X} +\sqrt{\Delta}-4/\mu$. Then $\cE$ is decomposable.
\end{corollary}

\begin{proof} By Bogomolov inequality and Lemma \ref{lem:estdelta=0}, we may assume that $\cE$ is not semistable.
By (\ref{eq:m+alpha}),  we have $ \sqrt{\Delta} \le -\rho$  and we conclude by Theorem \ref{prop:nssemistable}.
 \end{proof}

The last statement of this section shows that decomposability is determined by the pair $(\tau,\rho)$:

\begin{corollary}\label{cor:tau+rho}
Let $(X,\cE)$ be as in Setup \ref{ass:sec4}; then $\cE$ is decomposable if and only if $\tau +\rho=0$.
\end{corollary}

\begin{proof}
Assume that $\tau +\rho=0$; by Theorem \ref{prop:bound} we may assume that $\tau >0$, hence, by Lemma \ref{lem:alphabetass}, $\cE$ is not semistable. Consider a minimal covering family of rational curves $\cM$, of $H_X$-degree $\mu$. By Theorem \ref{prop:nssemistable} the bundle is decomposable unless possibly if $i_X=2$ and $\mu=1$ (hence $\cM$ is unsplit).
In this case, by Remark \ref{rem:idea},  $\cE$ is uniform, and the decomposability follows from  Corollary \ref{cor:splitsection}. The converse is a direct computation.
 \end{proof}

%%%%%%%%%%%%%%%%%%%%%%%%%%%%%%%%%%%%%%%%%%%%%%%%%%

\section{Fano bundles}\label{sec:fano}
%%%%%%%%

In this section we will apply our techniques to Fano bundles, i.e. bundles whose projectivization is a Fano manifold, in order to obtain structure theorems and partial classification results. Notice that a bundle $\cE$ on $X$ is a Fano bundle if and only if $\tau < i_X$

We begin by describing the second contraction $\f:\P(\cE) \to Y$. We will say that $\f$ is a $\P^1${\em -bundle}  if there is a rank two vector bundle $\cF$ on $Y$ such that $\P(\cE) = \P(\cF)$; we will say that $\f$ is a {\em conic bundle} if there is a rank three vector bundle $\cF$ on $Y$ such that $\P(\cE)$ embeds in $\P(\cF)$ as a divisor of relative degree two. We will denote by $l(R_2)$ the length of the extremal ray $R_2$, i.e. the minimum anticanonical degree of rational curves whose numerical class is in $R_2$ and by $C$ a curve in the ray whose anticanonical degree equals the length. The length of a $\P^1${\em -bundle} contraction is  two, while the length of a conic bundle contraction is either one if there are singular fibers, or two if all the fibers are reduced and irreducible.

\begin{lemma}\label{lem:contractiontype} Let $(X,\cE)$ be as in Setup \ref{ass:sec4} and assume that $\cE$ is Fano. If $\cE$ is indecomposable then, with the same notation as above:
%then the second contraction $\f:\P(\cE) \to Y$, its length and the Fano threshold of $\cE$ are
\begin{enumerate}
 \item either $\f$ is a $\P^1$-bundle, $ l(R_2)= 2$, $\tau=i_X-\frac{2}{H \cdot C}$, or
 \item $\f$ is a conic bundle with reducible fibers, $ l(R_2)= 1$, $\tau=i_X-\frac{1}{H \cdot C}$, or
\item $\f$ is the blow-up of a codimension two smooth subvariety, $l(R_2)= 1$, $\tau=i_X-\frac{1}{H \cdot C}$.
\end{enumerate}
In all cases $Y$ is smooth and Fano.
\end{lemma}

\begin{proof}
Observe first that any non trivial fiber of $\f$ is one dimensional. In fact, if a fiber contained a surface we would have $\Delta=\tau^2$ by Lemma \ref{lem:newns}, and $\cE$ would split by Remark \ref{rem:extremedelta}. By \cite[Thm. 1.2]{Wis} we then get that $Y$ is smooth  and either $l(R_2)=1$ and we are in case (2), (3) or $l(R_2)=2$ and $\f:\P(\cE) \to Y$ is a conic bundle without reducible fibers. In this case we conclude by Lemma \ref{lem:pbundle}. Moreover $Y$ is always Fano, since it is covered by rational curves and its Picard number is one.
  \end{proof}

%\begin{lemma}\label{lem:pbundle}
%Let $X$ be a Fano manifold, and $\f: X\to Y$  a conic bundle without reducible fibers. Then $\f: X\to Y$ is a
%$\P^1$-bundle.
%\end{lemma}
%
%\begin{proof} Let $f$ be a fiber of $\f$. The restriction $i^*:H^2(X, \Z) \to  H^2(f,\Z)$ is surjective by \cite[Rem. 2 after Thm. 6]{Borel}; therefore there exists a class in $H^2(X,\Z)$ which restricts to the class of a point in $f$. Since $X$ is Fano, this class corresponds to the
%class of a line bundle $L$ in $\Pic(X)$ which restricts to $\cO(1)$ on $f$ and we conclude by \cite[Lemma 2.12]{F}. \end{proof}

\begin{lemma}\label{lem:pbundle}
With the same notation as above, assume that $\f:\P(\cE)\to Y$ is a smooth conic bundle. Then $\f$ is a $\P^1$-bundle.
\end{lemma}

\begin{proof}
Denote by $-K'$ the relative anticanonical divisor of the morphism $\f:\P(\cE)\to Y$, by $H'$ the pull-back via $\f$ of $H_Y$, the ample generator of $\Pic(Y)$ and by $i_Y$ the index of $Y$. Let $f$ and $f'$ be the numerical class of the fibers of $\pi$ and $\f$, respectively, and set $\mu:=H\cdot f'$, $\mu':=H'\cdot f$. For simplicity, we will denote $\nu:=i_X\mu-2=K\cdot f'$, $\nu':=i_Y\mu'-2=K'\cdot f$. We then have the following intersection numbers:

\begin{table}[h!]
\begin{tabular}{|c|c|c|c|c|}
\hline
&$-K$&$H$&$-K'$&$H'$\\
\hline
$f$&$2$&$0$&$-\nu'$&$\mu'$\\
\hline
$f'$&$-\nu$&$\mu$&$2$&$0$\\
\hline
\end{tabular}
 \end{table}

that allow us to write:

\begin{equation}\label{eq:basechange}
 \begin{pmatrix}-K\\H\end{pmatrix}=
 \begin{pmatrix} -\frac{\nu}{2} &\frac{4-\nu\nu'}{2\mu'}\\ \frac{\mu}{2} & \frac{\nu'\mu}{2\mu'} \end{pmatrix}
 \begin{pmatrix}-K'\\H'\end{pmatrix}
 %=:A \begin{pmatrix} -K \\H \end{pmatrix}.
\end{equation}

Assume that $\f$ is not a $\P^1$-bundle. Equivalently, $\{-K',H'\}$ is a base of $\Pic(\P(\cE))$, therefore the determinant of the base change matrix is $\pm 2$, and we get $\mu=2\mu'$.
% and we may write:
%
%\begin{equation}\label{eq:basechange}
% \begin{pmatrix}-K\\H\end{pmatrix}=
% \begin{pmatrix} -\frac{\nu}{2} &\frac{4-\nu\nu'}{2\mu'}\\ \mu' & \nu' \end{pmatrix}
% \begin{pmatrix}-K'\\H'\end{pmatrix}
% %=:A \begin{pmatrix} -K \\H \end{pmatrix}.
%\end{equation}
Using this base change, together with the equality 
$$
-\Delta=\tau^2\tan^2\left(\dfrac{\pi}{n+1}\right)=\left(\dfrac{\nu}{\mu}\right)^2\tan\left(\dfrac{\pi}{n+1}\right),
$$
obtained as in Proposition \ref{prop:notbig},
the Chern-Wu relation $K^2=\Delta H^2$ transforms into:
\begin{equation}\label{eq:CW-Y}
K'^2=2\dfrac{\nu\nu'(1+a)-4}{\nu\mu'(1+a)}K'H'-\dfrac{16-8\nu\nu'+\nu^2\nu'^2(1+a)}{\nu^2\mu'^2(1+a)}H'^2.
\end{equation}
where $a$ denotes $\tan^2(\pi/(n+1))$.

It is known (see, for instance, \cite{S}) that being $\f$ a conic bundle, the cohomology class of the closed set of points in $Y$ whose inverse image by $\f$ is a singular conic is $-\f'_*(K'^2)$. By hypothesis, $\f$ is smooth, hence the intersection number $K'^2H'^{n-1}$ is equal to zero. Using (\ref{eq:CW-Y}) and the obvious equalities $H'^{n+1}=0$, $-K'H'^{n}=2H_Y^{n}$, we then have
$$
\nu\nu'=\dfrac{4}{1+a}=4\cos\left(\dfrac{\pi}{n+1}\right)=1,2,3,\quad\mbox{for }n=2,3,5\mbox{, respectively.}
$$

Since $\nu=i_X\mu-2=2(i_X\mu'-1)$ is even, we conclude that the only possibility is $n=3$, $i_X=2$, $i_Y=3$, and $\mu'=1$.
In this case, we have $-K=K'+H'$, $H=-K'+H'$, and the Chern-Wu relation reads as $K'^2=-H'^2$, from which we obtain
$-KH^3=-4K'H'^3$. In particular $X$ has degree $8$, contradicting that it is a del Pezzo $3$-fold of index two (cf. \cite{F2}).   

\end{proof}
%
%\begin{remark}\label{rem:borel}
%In previous versions of this paper, Lemma \ref{lem:pbundle} followed from a result of Borel (cf. \cite[Rem. 2 after Thm. 6]{Borel}), from which we obtained that every smooth conic bundle on a simply connected manifold $M$ (in particular on a Fano manifold) is a $\P^1$-bundle. We have recently learned from C. Casagrande that Borel's proof is not correct, although it would work if we further assumed that $H^3(M,\Z)$ is torsion free. 
%\end{remark}

We will now show that, with one exception, indecomposable Fano bundles are stable. Note that it was conjectured
by Grauert and Schneider (actually they provided an incomplete proof, cf. \cite{GS}) that every indecomposable rank two vector bundle on
$\P^n$, $n\geq 4$ is semistable.

\begin{theorem} \label{prop:stFano}
Let $(X,\cE)$ be as in Setup \ref{ass:sec4} and assume that $\cE$ is Fano, indecomposable and not stable; then  $X \simeq \P^2$ and $\cE$ is a bundle with $c_1=0, c_2=1$ whose projectivization is the blow-up of a smooth three-dimensional quadric along a line.
\end{theorem}

\begin{proof}
Let $\cM$ be a minimal covering family of rational curves on $X$ of $H_X$-degree equal to $\mu$.
Assume first that $\cE$ is not semistable; by Proposition \ref{prop:alpha=beta} we have $\rho <0$, hence,
by Theorem  \ref{prop:nssemistable}, $\cE$ is  decomposable unless possibly when  $\rho=-1$, $\mu=1$ and $i_X=2$. In this case, by Remark \ref{rem:idea}, $2=i_X>\tau\geq\tau(\ell)\geq 1$ for every $\ell\in\cM$, hence $\cE$ is uniform of splitting type $(-1,0)$ and we may conclude by Corollary \ref{cor:splitsection}.

Assume now that $\cE$ is semistable. By Remark \ref{rem:alphanotinteger} we have $\rho=0$. Let $C$ be a minimal rational curve generating the second extremal ray of $\P(\cE)$. A divisor $E \in |L|$ is not nef, otherwise we have $\tau=0$ and $\cE$ splits by Theorem \ref{prop:bound}.
In particular $E$ has negative intersection with $C$,  hence the second contraction $\f:\P(\cE) \to Y$ is a smooth blow-up with exceptional locus $E$, by Lemma \ref{lem:contractiontype}; in particular $E\cdot C=-1$.

 On the other hand $\beta=0$ implies that $c_2>0$, so that $\Delta<0$ and we may apply Proposition \ref{prop:secondray} to get:

\begin{equation}\label{eq:notstable}
\frac{4c_2}{d}=-\Delta=\tau^2\tan^2\left(\dfrac{\pi}{2n}\right).
\end{equation}
As in the proof of Proposition \ref{prop:notbig}, we now use  \cite{Ni}, \cite[Prop. 2]{Ca} to get that $n$ is equal to $2$ or $3$.
Moreover, from $l(R_2)=1$ and $E\cdot C=-1$ we get that
either $\mu=1$ (hence $\tau=i_X-1$) and $i_X=3$ or $\mu=3$ and $i_X=1$.

If $n=3$, since a Fano threefold with Picard number one is covered either by lines or by conics, we may assume $i_X=3$.
Then $X$ is a smooth quadric, hence $d=2$, and, from (\ref{eq:notstable}) we get
$2c_2=\frac{4}{3}$, which is impossible.

If $n=2$ then $X$ is $\P^2$; from $\tau=2$ we get that $L+H$ is nef and trivial on the second ray, hence $L+H =\f^*H_Y$ for a line bundle $H_Y$ on $Y$; since $H_Y \cdot f=1$ this line bundle is the generator of $\Pic(Y)$; computing the canonical bundle of $\P(\cE)$ with the blow-up formula we get that $i_Y=3$, i.e. $Y \simeq \Q^3$. The exceptional divisor $E$ contains a fiber of $\pi$, since $c_2=1$ and $\cE$ is semistable.
The center of the blow-up contains the image in $Y$ of this fiber, which is a line $\ell$ in $\Q^3$; since the center is smooth then the center is $\ell$.
 \end{proof}

The main result of this section is a characterization of bundles $\cE$ whose projectivization  $\P(\cE)$ has a second contraction which is a $\P^1$-bundle. We will first recall the following:

\begin{example} Let $K(G_2)$ be a $5$-dimensional Fano homogeneous contact manifold of type $G_2$.  $K(G_2)$ is a Fano manifold with Picard number one, index three and $b_4=1$ which is a linear section of the grassmannian $\G(1,6)$ with a $\P^{13}$ (see \cite[Ex. 1]{Mukai}). Since $K(G_2)$ is covered by lines, the restriction of the universal quotient bundle $\cQ$ on $\G(1,6)$ is a Fano bundle on $K(G_2)$, with $\tau=1$.
The projectivization of $\cQ|_{K(G_2)}$ has a second $\P^1$-bundle structure, over a five dimensional smooth quadric $\Q^5$, which corresponds to the projectivization of a Cayley bundle $\cC$, see \cite[1.3]{O}; via this description  $\P(\cQ|_{K(G_2)}) \to K(G_2)$ can be seen as the universal family of jumping lines of $\cC$.
\end{example}

\begin{theorem} \label{thm:pbundles}
Let $(X,\cE)$ be as in Setup \ref{ass:sec4}, and assume that $\cE$ is indecomposable. Then the following are equivalent
\begin{enumerate}
\item $\P(\cE)$ admits an unsplit dominating family $\cM'$ of rational curves of positive $H$-degree;
\item $\P(\cE)$ has a second contraction which is a $\P^1$-bundle;
\item $(X,\cE)$ is one of the following
\begin{enumerate}
\item $(\P^{2}, T_{\P^{2}})$;
\item $ (\P^3, \cN)$, with $\cN$ a null-correlation bundle;
\item $ (\Q^3, \cS)$ with $\cS$ the restriction of a spinor bundle;
\item $ (\Q^5, \cC)$ with $\cC$ a Cayley bundle;
\item $ (K(G_2), \cQ)$, with $\cQ$ the restriction of the universal quotient bundle.
\end{enumerate}
\end{enumerate}
\end{theorem}

\begin{proof} (1)\,$ \Rightarrow$ (2).\quad
By Remark \ref{cor:nsapp} the dimension of every irreducible component of $\Cloc(\cM')_y$ is one, for every $y \in \P(\cE)$. By \cite[Prop.~1]{BCD}, this implies that the quotient of $\P(\cE)$ by the $\cM'$-equivalence relation is a morphism $\f:\P(\cE)\to Y$. By Kleiman's criterion $\P(\cE)$ is a Fano manifold; we conclude by Lemma \ref{lem:contractiontype} observing that the length of the ray contracted by $\f$ is two.\par
\medskip
(2)\,$ \Rightarrow$ (3).\quad
Assume that $\f:\P(\cE) \to Y$ makes $\P(\cE)$ a $\P^1$-bundle over a smooth (Fano) variety $Y$. Let $\cE'$ be the normalized rank two vector bundle on $Y$ whose projectivization is $\P(\cE)$, and $c_1'$ be its first Chern class. Denote by $L'$ a divisor associated with its tautological line bundle, and by $H'$ the pull-back via $\f$ of $H_Y$, the ample generator of $\Pic(Y)$. Denote $d_X:=H_X^n$, $d_Y:=H_Y^n$. Finally take $f$ and $f'$ to be fibers of $\pi$ and $\f$ respectively, and set $\mu:=H\cdot f'$, $\mu':=H'\cdot f$, so that $\tau=\tau(\cE)=i_X-2/\mu$, $\tau':=\tau(\cE')=i_Y-2/\mu'$.

Using the intersection numbers of $H,L,H'$ and $L'$ with $f$ and $f'$ we may easily write:
$$
\left\{\begin{array}{l}\vspace{0.2cm}
H'=-\dfrac{\mu'}{2}(c_1-\tau)H+\mu' L\\
L'=\left(-\dfrac{\mu'}{4}(c_1-\tau)(c_1'-\tau')+\dfrac{1}{\mu}\right)H+\dfrac{\mu'}{2}(c'_1-\tau') L
\end{array}\right.
$$
Since $\{H,L\}$ and $\{H',L'\}$ are $\Z$-bases of $\Pic(\P(\cE))$ it follows that the determinant of the matrix of base change, which is $\mu'/\mu$, is $\pm 1$, hence $\mu=\mu'$. In particular we may write $H'=\frac{\mu}{2}(-K_{\rel}+\tau H)$, so that:
$$
\dfrac{d_Y}{d_X}=\left(\dfrac{\mu}{2}\right)^n\dfrac{(-K_{\rel}+\tau H)^n\cdot H/\mu}{-K_{\rel}\cdot H^n/2}=\left(\dfrac{\mu}{2}\right)^{n-1}\dfrac{\im\left((\tau+i\sqrt{-\Delta})^n\right)}{\sqrt{-\Delta}}.
$$
The last equality follows from our computations in Section \ref{sec:posit} and the negativity of $\Delta$ obtained by Proposition \ref{prop:secondray}. Furthermore, as in Proposition \ref{prop:notbig}, we have $n=2,3$ or $5$ and $\sqrt{-\Delta}=\tau\tan(\pi/(n+1))$. Operating in the expression above we get:
$$
\dfrac{d_Y}{d_X}=\left(\dfrac{\tau\mu}{2\cos(\pi/(n+1))}\right)^{n-1},
$$
and this, together with the equation obtained exchanging the roles of $X$ and $Y$, provides:
\begin{equation}\label{eq:final}
\left(i_X\mu-2\right) \left(i_Y\mu-2\right) =\tau\tau'\mu^2=
\begin{cases} 1& \text{if~} n=2 \\
2 &\text{if~} n=3\\
3 &\text{if~} n=5\end{cases}
\end{equation}
From this data we may easily obtain (up to exchanging $X$ and $Y$) the following values:
$$
\begin{array}{|l|l|l|l|l|l|l|l|l|}
\hline
\hspace{0.1cm}n\hspace{0.1cm}&\hspace{0.1cm}i_X\hspace{0.1cm}&
\hspace{0.1cm}i_Y\hspace{0.1cm}&\hspace{0.1cm}d\hspace{0.1cm}&
\hspace{0.1cm}\mu\hspace{0.1cm}&\hspace{0.1cm}\tau\hspace{0.1cm}&
\hspace{0.1cm}\Delta\hspace{0.1cm}&\hspace{0.1cm}c_1\hspace{0.1cm}&
\hspace{0.1cm}c_2\hspace{0.1cm}\\\hline
2&3&3&1&1&1&-3&-1&1\\\hline
3&4&3&1&1&2&-4&0&1\\\hline
5&5&3&1&1&3&-3&-1&1\\\hline
\end{array}
$$
Since the tangent bundle of $\P^2$, the null-correlation bundle on $\P^3$ and the Cayley bundle on $\mathbb Q^5$ are determined, among stable bundles, by their Chern classes (cf. \cite[8.1]{H}, \cite[Lem.~4.3.2]{OSS}, \cite{O}) and we know that $\cE$ is stable by Theorem \ref{prop:stFano} the implication follows.\par
\medskip
(3)\,$ \Rightarrow$ (1).\quad The family of the fibers of the second contraction of $\P(\cE)$ is unsplit, dominating and has positive $H$-degree.
 \end{proof}

\begin{corollary} \label{cor:pbundles}
Let $(X,\cE)$ be as in Setup \ref{ass:sec4} and assume that $\cE$ is Fano and indecomposable. Then the following are equivalent
\begin{enumerate}
\item $(X,\cE)$ is as in {\rm(3)} of Theorem \ref{thm:pbundles}.
\item $i_X-c_1 \equiv 0~ (\text{mod}~2)$.
\item $\tau <i_X -1$.
\end{enumerate}
\end{corollary}

\begin{proof}
(1)\,$ \Rightarrow$ (2) is a direct computation. From (2) it follows that $-K_{\P(\cE)}\cdot C=2L\cdot C+(i_X-c_1)H\cdot C$ is even, and we conclude (3) by Lemma \ref{lem:contractiontype}. By Lemma \ref{lem:contractiontype} again we have that, if (3) holds, then $\tau<i_X-1\leq i_X-1/H\cdot C$. This implies that $\P(\cE)$ has a second $\P^1$-bundle structure and we get (1) by Theorem \ref{thm:pbundles}. \end{proof}

As a consequence we get the following classification of uniform rank two vector
bundles on Fano manifolds.

\begin{corollary}\label{cor:unif}
Let $(X,\cE)$ be as in Setup \ref{ass:sec4} and let $\cM$ be a
covering unsplit family of rational curves on $X$ such that $\cM_x$ is irreducible for a general $x\in X$. If $\cE$ is indecomposable and uniform with respect to $\cM$, then  $(X,\cE)$ is either $(\P^2,T_{\P^2})$, $(\Q^3,\cS)$ or $(K(G_2), \cQ)$.
\end{corollary}

\begin{proof}
Consider the family of minimal sections $\til\cM$ of $\P(\cE)$ over $\cM$.
By Theorem \ref{thm:pbundles} either $\loc(\til\cM)$ is a divisor, or $(X,\cE)$ is one of the pairs listed there. Moreover, checking uniformity in the classification, we get that in the second case $(X,\cE)$ is  $(\P^2,T_{\P^2})$, $(\Q^3,\cS)$ or $(K(G_2), \cQ)$.

In particular, it suffices to show that $\cE$ splits whenever $\loc(\til\cM)$ is a divisor.  If this is the case, the irreducibility of $\cM_x$ implies that $\loc(\til\cM)$ is a unisecant divisor (cf. Remark \ref{lem:a=bprelim}), determined by an injection $s:\cO(b)\to\cE$. But the general element of $\cM$ does not meet the set of zeroes of $s$, hence by construction $b=(c_1+\tau(\ell))/2$. Note that we may assume that $\tau(\ell)> 0$, by \cite[Prop.~1.2]{AW}, then we may conclude by Corollary \ref{cor:splitsection}. \end{proof}

\begin{remark}\label{rem:watanabe}
Recently, using similar techniques, Watanabe has shown (cf. \cite{Wa}) that the only $\P^1$-bundles over Fano manifolds of Picard number one admitting a second smooth fibration of relative dimension one are those listed in our theorem. In his proof no assumption on $b_4$ is needed. 
\end{remark}

%%%%%%%%%%%%%%%%%%%%%%%%%%%%%%%%%%%%%%%%%%%%%%%%%%

\section{Applications to Hartshorne's Conjecture}\label{sec:apphart}

Throughout this section we will denote by $Y \subset \P^{n}$, with $n \ge 8$, a codimension two smooth subvariety. By Barth-Larsen Theorem (see, for instance, \cite[II, Thm.~3.2.1]{L}), setting $H_{\P^n}:= \cO_{\P^n}(1)$ we obtain that $\Pic(Y) \simeq \mathbb{Z}\langle (H_{\P^n})_{|Y} \rangle$ and $H^4(Y) \simeq \mathbb Z \langle (H_{\P^n})_{|Y}^2 \rangle$.
%Note also that we are not assuming that $Y$ is a Fano manifold (in fact if $Y$ is Fano then it is a complete intersection by \cite[Theorem 6]{BCh}).
 Let $\cN$ be the normal bundle of $Y$ in $\P^n$ and $\cE$ the bundle on $\P^{n}$ obtained from $Y$ via the Hartshorne-Serre correspondence, so that $Y\subset\P^n$ appears as the set of zeroes of a section $s\in H^0(\P^{n},\cE)$ and $\cE|_Y \simeq \cN$. \\
 We will apply the techniques previously developed to the pair $(\P^n,\cE)$ (and also to $(Y, \cN)$ although $Y$ is not necessarily Fano, see Remark \ref{rem:boundsarevalid} below). As usual, we will denote by $c_1$ and $c_2$ the integers determining the first and second Chern classes of $\cE$ (or of $\cN$), and by $\Delta:=c_1^2-4c_2$ the discriminant of $\cE$. Note that $\cE$ is not normalized: in fact, by adjunction and Kobayashi-Ochiai Theorem $c_1=(n+1)-i_Y\geq 2$.\\
Set $\tau:=\tau(\cE)$, $\rho:=\rho(\cE)$, denote by $\beta$ the minimum integer such that $\cE(\beta)$ has sections, and by $\tau_Y,\rho_Y$ and $\beta_Y$ the corresponding invariants of $\cN$.

\begin{lemma}\label{lem:rho_Y} Assume that $\Delta \ge 0$. Then $\rho\ge \rho_Y$.
\end{lemma}

\begin{proof}
We begin by taking $\ell$ to be a $1$-secant line to $Y$. The restriction $s_{|\ell}$ vanishes precisely at one point, so that we have an injection of bundles
$$
0\rightarrow\cO_\ell(1)\longrightarrow\cE_{|\ell},
$$
whose cokernel $\cE_{|\ell}\to\cO_\ell(c_1-1)$ provides a section $\tilde{\ell}$ of $\P(\cE)$ over $\ell$ verifying $\tau(\tilde{\ell})=-c_1+2\geq 0$. Since $\ell$ is movable and $c_1-1\geq 1$, then $\tilde{\ell}$ is movable and we may assert that $\rho\geq -c_1+2$.\\
Consider a nonzero section $t\in H^0(\P^n,\cE(\beta))$; by Proposition \ref{prop:alpha=beta} and Remark \ref{rem:alphanotinteger}, $\Delta\geq 0$ implies that $\rho=2\beta+c_1\le 0$.
Assume that $\rho< \rho_Y$; then $t$ must vanish on $Y$. It follows that $$c_2(\cE(\beta))\geq [Y]=c_2(\cE)=c_2(\cE(-c_1)).$$
Since $\beta\leq -c_1/2$ we may then say that $c_2(\cE(\beta))\geq c_2(\cE(-c_1))$ implies that $\beta\leq-c_1$, contradicting that $\rho\geq -c_1+2$.
%
%Assume the contrary and let $\beta':=(\rho'-c_1(\cE))/2$. Since $H^0(X,\cN_X)\neq 0$
%a section of $\cE(\beta')$ vanishes on $X$, and therefore
%$c_2(\cE(\beta')) \ge c_2(\cE)$, hence $\beta' \le -c_1(\cE)$; this
%implies that $\rho'  \le -c_1(\cE)$.
%Hence, for any line in $\P^{n+2}$ we have $\tau(\ell) \ge c_1(\cE)$;
%therefore the splitting of $\cE$ to a general line is $(a,b)$ with $a
%\le 0$ and $b \ge c_1(\cE)$.
%Let $\ell$ be a $1$-secant line to $X$. On this line $\cE$ must have a section
%vanishing only once, contradicting $c_1(\cE) >1$.
\end{proof}

\begin{corollary}\label{cor:c2normal}
With the same notation as above, assume that $c_2(\cN(\beta_Y))=0$; then $\cE$ is decomposable.
\end{corollary}

\begin{proof}
If $c_2(\cN(\beta_Y))=0$, then $c_2(\cE(\beta_Y))=0$ and we may say that $\Delta\geq 0$; in particular $\cE$ is not stable and  $\beta\leq -c_1/2$. Since moreover $c_2(\cE(\beta))\geq 0$, it follows that $\beta\leq\beta_Y$. On the other hand, Lemma \ref{lem:rho_Y} tells us that $\beta=(\rho-c_1)/2\geq (\rho_Y-c_1)/2\geq \beta_Y$, so that $\beta=\beta_Y$ and so $c_2(\cE(\beta))= 0$. It follows that $\cE$ splits.\end{proof}

\begin{remark}\label{rem:boundsarevalid}{\rm
On a Fano manifold of Picard number one the splitting of a vector bundle $\cE$ reduces to $c_2(\cE(\beta))=0$ via Kodaira Vanishing; Corollary \ref{cor:c2normal} shows that the same holds for the pair $(Y,\cN)$. Since this was the only use of assuming $X$ to be Fano in Section \ref{sec:posit}, this corollary  implies, in particular, that the bounds obtained in Section \ref{sec:posit} work for the pair $(Y,\cN)$ as well. 
%
%Let us observe that in Section \ref{sec:posit} the hypothesis on $X$ to be Fano in \ref{notn:setup} is used in Section 4 to prove that $c_2(\cE(\beta))=0$ implies the splitting of $\cE$
%% via the sequence (\ref{eq:shse}). 
%This splitting for the vector bundle $\cN$ is provided by Corollary \ref{cor:c2normal}, so that all bounds proven in Section 4 work for the pair $(Y,\cN)$.
}
\end{remark}

\begin{remark}\label{rem:dual}{\rm Let us observe that, being globally generated, $\cN(-1)$ is nef, hence $\tau_Y \leq c_1-2$ (it is well-known that, under our assumptions, $\cN(-1)$ is ample, so that the inequality is strict, in fact).}
\end{remark}

Let us assume that $\Delta >0$. In 
%\cite[Rmk 4.5]{HS} and 
\cite[Cor. 7.3]{Ho} it is shown that if $\cE$ is indecomposable, then $c_1> 2\left(n+\sqrt{n+1}\right)$. One of the ingredients of the proof is the fact that $c_2(\cE(\beta))\geq n+2$; joining this with our results of Section \ref{sec:posit} we are able to find a lower bound for $c_1$ of order $3/2$. 

\begin{proposition} If $\Delta >0$ and $c_1 \leq \sqrt{(n^2-4)(n-3)}+\sqrt{\Delta}+3$ then $\cE$ is decomposable.
\end{proposition}

\begin{proof}
By Lemma \ref{lem:posit4} we know that $\rho_Y^2 \le \Delta + 2\varepsilon \sqrt{\Delta} + \varepsilon^2$, hence
$$ \varepsilon^2 + 2\varepsilon \sqrt{\Delta} - 4c_2(\cN(\beta_Y)) \ge 0,$$
so we must have $\varepsilon \ge  -\sqrt{\Delta} + \sqrt{\Delta + 4c_2(\cN(\beta_Y))}$.
We can rewrite this as
$$\varepsilon \ge \dfrac{4c_2(\cN(\beta_Y))}{ \sqrt{\Delta} + \sqrt{\Delta + 4c_2(\cN(\beta_Y))}}.
$$

We have the following bounds:
$$
\varepsilon > \begin{cases}
 \varepsilon_1=\sqrt{c_2(\cN(\beta_Y))} & \quad \text{if} \quad \Delta \le c_2(\cN(\beta_Y))\\
 & \\
  \varepsilon_2=\dfrac{c_2(\cN(\beta_Y))}{\sqrt{\Delta}} &  \quad \text{if} \quad\Delta \ge c_2(\cN(\beta_Y))
\end{cases}
$$
Writing $\varepsilon$ as
$$\varepsilon=  \dfrac{2\sqrt{\Delta}}{\left(1+\dfrac{2 \sqrt{\Delta}}{\tau_Y-\sqrt{\Delta}}\right)^{n-2}-1}$$
we see, taking different terms of the binomial expansion of the denominator, that
 $$\varepsilon < \delta_1= \dfrac{2\sqrt{\Delta}}{(n-2)\left(\dfrac{2\sqrt{\Delta}}{(\tau_Y-\sqrt{\Delta})}\right)}=\dfrac{\tau_Y -\sqrt{\Delta}}{n-2}$$
$$\varepsilon < \delta_2=  \dfrac{2\sqrt{\Delta}}{\dfrac{(n-2)(n-3)}{2}\left(\dfrac{4 \Delta}{(\tau-\sqrt{\Delta})^2}\right)}= \dfrac{(\tau_Y-\sqrt{\Delta})^2}{(n-2)(n-3) \sqrt{\Delta}}.$$
Now let us write the conditions $\delta_i > \varepsilon_i$ and Remark \ref{rem:dual}:
$$ c_1-2>\tau_Y > \sqrt{\Delta} +(n-2) \sqrt{c_2(\cN(\beta_Y))} $$
$$ c_1-2>\tau_Y> \sqrt{\Delta}+ \sqrt{(n-2)(n-3) c_2(\cN(\beta_Y))}$$
Now we observe that, by Lemma \ref{lem:rho_Y} we have $\beta_Y \le \beta$, hence $c_2(\cN(\beta_Y)) \ge c_2(\cE(\beta))$, and we use \cite[Proposition 6.3]{Ho} to get $c_2(\cE(\beta)) \ge n+2$ and conclude.
\end{proof}

Finally let us consider separately the case in which  $Y$ is a {\it numerical complete intersection of type $(a,b)\in\Z^ 2$}, $a \leq b$, that is when $\Delta =(b-a)^2\geq 0$. 

\begin{remark}\label{rem:numericalbeta}{\rm Observe that, by Corollary \ref{cor:c2normal}, a numerical complete intersection with $\beta_Y=-a$ or $-b$  is a complete intersection.}
\end{remark}
%
%We can get the following condition on $Y$ to be a complete intersection.

\begin{corollary}\label{cor:nci} Assume that $Y$ is a numerical complete intersection of type $(a,b)$ with $a \le b$.
%  and set $k:=\frac{b-1}{a-1}$
%If  $a < (n-2)k +1$,
If $(n-3)(b-1)>(a-2)(a-1)$, then $Y$ is a complete intersection.
\end{corollary}

\begin{proof}
We will show that $\cE$ is decomposable. By Remark \ref{rem:dual} we have that $\tau_Y < c_1-2=a+b-2$. Note that, by Proposition \ref{prop:alpha=beta}, Remark \ref{rem:alphanotinteger} and  Remark \ref{rem:extremedelta2}, $\cE$ splits if and only if $\rho=2\beta+c_1=-\sqrt{\Delta}$. Since $\beta\in\Z$, using Lemma  \ref{lem:posit4}, this is equivalent to $\varepsilon<2$.\\
In case  $\Delta=0$ this gives $\tau_Y <2(n-2)$, which is satisfied in our assumptions. In case $\Delta >0$ we have
$$\varepsilon=\dfrac{2\sqrt{\Delta}(\tau_Y-\sqrt{\Delta})^{n-2}}{(\tau_Y+\sqrt{\Delta})^{n-2}-(\tau_Y-\sqrt{\Delta})^{n-2}}
<\dfrac{2\sqrt{\Delta}(c_1-2-\sqrt{\Delta})^{n-2}}{(c_1-2+\sqrt{\Delta})^{n-2}-(c_1-2-\sqrt{\Delta})^{n-2}}.$$
It follows that
$$\varepsilon<
\dfrac{2(b-a)(a-1)^{n-2}}{(b-1)^{n-2}-(a-1)^{n-2}}.$$
The condition $\varepsilon <2$  holds trivially when $a=1$, hence we may assume $b> a>1$ and a sufficient condition for the splitting is:
$$
1>\dfrac{(b-a)(a-1)^{n-2}}{(b-1)^{n-2}-(a-1)^{n-2}}=(a-1)\dfrac{\left(\frac{b-1}{a-1}\right)-1}{\left(\frac{b-1}{a-1}\right)^{n-2}-1},
$$
that is,
$$
\dfrac{\left(\frac{b-1}{a-1}\right)^{n-2}-1}{\left(\frac{b-1}{a-1}\right)-1}>(a-1).
$$
In particular, since $(b-1)/(a-1)> 1$, this condition holds whenever $1+(n-3)(b-1)/(a-1)>(a-1)$, that is $(n-3)(b-1)>(a-2)(a-1)$.
\end{proof}

\begin{remark}
Corollary \ref{cor:nci} improves the known results about Hartshorne's Conjecture for codimension two numerical complete intersections (Cf. \cite[Corollary 2.3]{EF}).
\end{remark}

\bibliographystyle{amsalpha}

\end{document}